\documentclass[reqno,12pt]{amsart}

\usepackage[colorlinks=true]{hyperref}
\hypersetup{urlcolor=blue, citecolor=red}

\usepackage{amssymb,graphicx,color}
\usepackage[latin1]{inputenc}

\theoremstyle{plain}
\newtheorem{thm}{Theorem}[section]
\newtheorem{lem}{Lemma}[section]
\newtheorem{prop}{Proposition}[section]
\newtheorem{cor}{Corollary}[section]
\newtheorem{defs}{Definition}[section]

\theoremstyle{definition}
\newtheorem{rmk}{Remark}[section]

\newcommand{\NN}{{\mathbb{N}}}

\newcommand{\RR}{{\mathbb{R}}}

\newcommand{\bfu}{\mathbf{u}}
\newcommand{\bfv}{\mathbf{v}}
\newcommand{\bfw}{\mathbf{w}}

\newcommand{\bff}{\mathbf{f}}

\newcommand{\bfx}{\mathbf{x}}

\newcommand{\bfA}{\mathbf{A}}
\newcommand{\bfB}{\mathbf{B}}

\newcommand{\bfF}{\mathbf{F}}

\newcommand{\bfnabla}{\boldsymbol{\nabla}}

\newcommand{\Ccal}{{\mathcal C}}

\newcommand{\Bcal}{{\mathcal B}}

\newcommand{\Pcal}{{\mathcal P}}

\newcommand{\Tcal}{{\mathcal T}}
\newcommand{\Ucal}{{\mathcal U}}

\newcommand{\rmd}{\mathrm{d}}
\newcommand{\rmw}{\mathrm{w}}

\newcommand{\cyl}{{\mathrm{cyl}}}

\newcommand{\fpsss}{{\mathrm{fpsss}}}
\newcommand{\wsss}{{\mathrm{wsss}}}

\newcommand{\inv}{{\operatorname{inv}}}
\newcommand{\average}[1]{\langle{#1}\rangle}
\newcommand{\dual}[1]{\langle{#1}\rangle}

\newcommand{\Lim}{\operatorname*{\textsc{Lim}}_{T\rightarrow \infty}}

\newcommand{\dontshow}[1]{}

\mathchardef\mhyphen="2D

\begin{document}
\numberwithin{equation}{section}


\title[Optimal minimax bounds for the Navier-Stokes equations]{Optimal minimax bounds for time and ensemble averages for the incompressible Navier-Stokes equations}

\author[R. Rosa]{Ricardo M. S. Rosa}

\author[R. Temam]{Roger M. Temam}

\address[Ricardo M. S. Rosa]{Instituto de Matem\'atica, Universidade Federal do Rio de Janeiro, Brazil}
\address[Roger M. Temam]{Indiana University, Bloomington, IN, USA}

\email[R. Rosa]{rrosa@im.ufrj.br}
\email[R. Temam]{temam@indiana.edu}

\date{\today}

\subjclass[2020]{76D05, 76D06, 35Q30, 37L99}
\keywords{dissipative systems, Navier-Stokes equations, mean bounds, ensemble averages, time averages, stationary statistical solutions, invariant measures}

\begin{abstract}
  Obtaining sharp estimates for quantities involved in a given model is an integral part of the modeling process. For dynamical systems whose orbits display a complicated, perhaps chaotic, behaviour, the aim is usually to estimate time or ensemble averages of given quantities. This is the case, for instance, in turbulent flows. In this work, the aim is to present a minimax optimization formula that yields optimal bounds for time and/or ensemble averages for the two- and three-dimensional Navier-Stokes equations. The results presented here are extensions to the infinite-dimensional setting of a recent result on the finite-dimensional case given by Tobasco, Goluskin, and Doering in 2017. The optimal result occurs in the form of a minimax optimization problem and does not require knowledge of the solutions, only the law of the system. The minimax optimization problem appears in the form of a maximization over a portion of the phase space of the system and a minimization over a family of auxiliary functions made of cylindrical test functionals defined on the phase space. The function to be optimized is the desired quantity plus the duality product between the law of the system and the derivative of the auxiliary function.
\end{abstract}

\maketitle

\begin{center}
\parbox{5.6in}{\emph{This work is dedicated to the memory of Ciprian Foias, with whom we developed much of the material on which this article is built. Ciprian was a friend, a long time collaborator, and a great source of motivation and inspiration. The pleasant walks to the coffee shops will not be the same.}}
\end{center}

\section{Introduction}

In this work, we extend to an infinite dimensional setting a recent result on optimal bounds of mean quantities of evolutionary differential equations proved by Tobasco, Goluskin and Doering in \cite{TobGolDoe2017}. We consider a general framework and apply our results to the incompressible Navier-Stokes equations, which is a fundamental model of fluid flows. Both two- and three-dimensional Navier-Stokes equations are considered. Other models will be addressed in a forthcoming work \cite{RTinpreparation}.

The result in \cite{TobGolDoe2017} concerns the \emph{optimality} of a minimax-type estimate for \emph{the asymptotic time average} of quantities defined on \emph{individual solutions} of the system. Their proof also shows that this optimality holds for averages with respect to \emph{invariant measures} of the associated dynamical system.

The work in \cite{TobGolDoe2017} follows a series of recent papers in which similar estimates are used, in combination with semidefinite programming \cite{Parrilo2003}, to obtain Lyapunov functions for dynamical systems and address many dynamical-system problems in a computationally tractable way, as long as the equations are given in terms of polynomial functions. We refer the reader to the review article \cite{CheGouHuaPa2014}, which has many applications of this technique, including upper bound estimates (not necessarily optimal) associated with a number of fluid flow problems (based on the finite-dimensional system obtained from the Galerkin approximation of the Navier-Stokes equations).

Consider a differential equation $u_t = F(u)$ generating a semiflow $(t, u_0) \mapsto u(t;u_0)$, defined on a phase space $X$ and for times $t\geq 0$, where $u(t,u_0)$ is the unique global solution starting at $u(0;u_0) = u_0$. Suppose we want to estimate the asymptotic time average of a given quantity $\phi : X \rightarrow \RR$, for solutions starting at a given compact invariant set $B\subset X$. The time average may not converge, but we may look for an upper bound for its superior limit:
\begin{equation}
  \label{limsuptimeave}
  \sup_{u_0\in B} \limsup_{T\rightarrow \infty} \frac{1}{T}\int_0^T \phi(u(t; u_0)) \;\rmd t.
\end{equation}

The upper bound estimates for \eqref{limsuptimeave} given in \cite{CheGouHuaPa2014} and in the references therein are obtained by adding to the desired quantity an auxiliary term and then minimizing the result over the auxiliary functions used to generate this term. Notice that, in the finite-dimensional case, for any $\Psi\in \Ccal^1(B)$, which stands for the space of real-valued continuously differentiable functions defined on $B$, we have
\begin{equation}
  \label{introdphidt}
  \frac{\rmd}{\rmd t}\Psi(u(t;u_0)) = \dual{F(u(t;u_0)), \Psi'(u(t;u_0))}.
\end{equation}
Since $B$ is compact and invariant, the function $\Psi(u(t;u_0))$ is bounded, and we have, in view of \eqref{introdphidt},
\[ \limsup_{T\rightarrow \infty} \frac{1}{T}\int_0^T \dual{F(u(t;u_0)), \Psi'(u(t;u_0))} \;\rmd t = 0,
\]
for an arbitrary $\Psi\in \Ccal^1(B)$. This yields the identity
\begin{multline}
 \label{individual_bound_intro}
 \sup_{u_0\in B} \limsup_{T\rightarrow \infty} \frac{1}{T}\int_0^T \phi(u(t; u_0)) \;\rmd t \\
    = \sup_{u_0\in B} \inf_{\Psi\in \Ccal^1(B)} \limsup_{T\rightarrow \infty} \frac{1}{T}\int_0^T \left\{\phi(u(t; u_0)) + \dual{F(u(t;u_0)), \Psi'(u(t;u_0))}\right\}\;\rmd t.
\end{multline}

Since $B$ is invariant, we bound the limsup in \eqref{individual_bound_intro} by the maximum of the integrand on the compact set $B$, obtaining
\begin{equation}
 \label{leq_bound_intro}
 \sup_{u_0\in B} \limsup_{T\rightarrow \infty} \frac{1}{T}\int_0^T \phi(u(t; u_0)) \;\rmd t \leq  \inf_{\Psi\in \Ccal^1(B)} \max_{u_0\in B} \left\{ \phi(v) + \dual{F(v), \Psi'(v)}  \right\}.
\end{equation}

The remarkable result proved by Tobasco, Goluskin and Doering in \cite{TobGolDoe2017} is that the estimate \eqref{leq_bound_intro} is, in fact, optimal, which is obtained by establishing the identity
\begin{equation}
 \label{optimal_intro}
 \max_{u_0\in B} \limsup_{T\rightarrow \infty} \frac{1}{T}\int_0^T \phi(u(t; u_0)) \;\rmd t = \inf_{\Psi\in \Ccal^1(B)} \max_{u_0\in B} \left\{ \phi(v) + \dual{F(v), \Psi'(v)}  \right\}.
\end{equation}

There are two crucial steps in their result. One is that the maximum on the left hand side of \eqref{optimal_intro} can be written in terms of an ergodic measure. After this step, it is not difficult to write a version of \eqref{individual_bound_intro} with Borel probability measures instead of a limsup of time averages. Then, the second crucial step is to use a suitable minimax theorem, namely Sion's minimax theorem (see Theorem \ref{thmsionminimax} below), to prove the identity in the passage from max-inf to inf-max, where the maximum is over Borel probability measures. After that, it is straightforward to use that the maximum over Borel probability measures is achieved on Dirac delta measures, hence corresponding to the maximum over points in $B$.

A related problem is to find the orbit $u(t;u_0^*)$ (or orbits), or the associated invariant measure $\mu^*$ (or measures), generated by the limit of the time averages of this orbit, that attains the maximum above. This is a topic in Ergodic Optimization that started in the early 1990's. See Jenkinson's paper \cite{Jenkinson2019}.

The difficulty of extending this to infinite-dimensions is three-fold: the compactness of $B$ needed in the representation theorem associating asymptotic time averages with invariant measures; the compactness of $B$ needed in the application of Sion's Minimax Theorem; and the equivalence between the classical notion of (Lagrange) invariance of a measure (i.e. $\mu(S(t)^{-1}E)=\mu(E)$, for any measurable set $E$ and any $t\geq 0$) and an Eulerian-type invariance (i.e. $\int_B \dual{F(v),\Psi'(v)}\;\rmd\mu(v)=0$, for any $\Psi\in \Ccal^1(B)$).

The compactness of the starting set $B$ can be avoided by the assumption of the system to be dissipative, in the sense of the existence of a compact pointwise attracting set for the system. And the equivalence between the two notions of invariance holds provided the system can be approximated by a continuously differentiable left-invertible semigroup of operators with suitable properties. These conditions are expected to hold for many infinite-dimensional systems, as illustrated here with the two-dimensional Navier-Stokes equations.

The difficulty of extending \eqref{optimal_intro} to the three-dimensional Navier-Stokes equations, however, is of a different nature. It lies on the fact that the three-dimensional version of this system is not known to be well-posed. There is no known semigroup at our disposal. In order to overcome such a difficulty, we resort to the concept of stationary statistical solution of the Navier-Stokes equations, which was introduced by Foias and Prodi \cite{Foias72, Foias73, FoiasProdi1976}. In this case, we obtain some partial results. We obtain a possibly non-optimal minimax formula for \emph{Foias-Prodi stationary statistical solutions} and for the limsup of time averages of Leray-Hopf weak solutions. We also consider in this article a wider class of statistical solutions that we term \emph{weak stationary statistical solution}, and in this case we prove an optimal minimax estimate as above.

We end the introduction by summarizing our main results:
\begin{enumerate}
  \item \label{summaryextendtoinfdim} We extend the optimal minimax formula to a general infinite-dimensional setting that applies to dissipative continuous semigroups where the main ingredients are the existence of a compact pointwise attracting set and the equivalence between the classical notion of (Lagrangian-type) invariant measure and a suitable Eulerian-type notion;
  \item We apply our optimality result to invariant measures and asymptotic time averages of individual solutions of the two-dimensional Navier-Stokes equations;
  \item We apply our optimality result to a weak version of stationary statistical solutions of the three-dimensional Navier-Stokes equations;
  \item We obtain a possibly non-optimal upper bound formula for the Foias-Prodi stationary statistical solutions and for the time averages of Leray-Hopf weak solutions of the three-dimensional Navier-Stokes equations.
\end{enumerate}

In a forthcoming work \cite{RTinpreparation}, we go one step further in item \eqref{summaryextendtoinfdim} above by addressing conditions for which the Lagrangian and Eulerian notions of invariant measure coincide and then by obtaining the validity of the minimax optimality result in a variety of other examples of infinite-dimensional equations.

\section{Mathematical framework}

We work with two different frameworks, a more general one for \emph{(weak) stationary statistical solutions} and a more specific one for asymptotic time averages of orbits of \emph{semigroups}. The two concepts are explained in detail in this section.

In both of them, we let $X$ be a \emph{Hausdorff topological space} and let $W$ be a \emph{Banach space} with a \emph{continuous inclusion} $X\subset W'$. Typically, $W$ is continuously included in $X$. The norm in $W$ is denoted by $\|\cdot\|_W$, and the duality product between $W$ and $W'$, by $\average{\cdot,\cdot}_{W',W}$. Finally, we let $F : X \rightarrow W'$ be a \emph{Borel map} and consider the \emph{differential equation}
\begin{equation}
  \label{diffeq}
  \frac{\rmd u}{\rmd t} = F(u).
\end{equation}

At this point, we are not discussing conditions under which the differential equation \eqref{diffeq} has global solutions or not, and whether these solutions are uniquely determined by the initial condition or not. These results are left to each application and depend on the structure of $F$ and on further properties of the functional spaces. For a number of examples, see, for instance, \cite{BabinVishik1992, SellYou2002, Temam1988}.

\subsection{Weak stationary statistical solutions}
\label{subsecsss}

For the first framework, we introduce the following definitions, stemming from \cite{Foias72,Foias73,vishikfursikov79, vishikfursikov88,FoiTem75, FMRT2001,FRT2019, BMR2016}.
\begin{defs}
  \label{defscyl}
  A \textbf{cylindrical test functional} on the dual $W'$ of a Banach space $W$ is any function $\Psi:W'\rightarrow \RR$ of the form
  \[ \Psi(u) = \psi(\dual{u,w_1}_{W',W}, \ldots, \dual{u,w_m}_{W', W}), \qquad \forall u\in W',
  \]
  where $w_1, \ldots, w_m\in W$, $m\in \NN$, and $\psi\in \Ccal_c^1(\RR^m)$. We denote the set of such test functions as $\Tcal^\cyl(W')$.
\end{defs}

As usual, $\Ccal_c^1(\RR^m)$, in the definition above, denotes the space of real-valued compactly supported continuously differentiable functions on $\RR^m$.

Any test function $\Psi\in \Tcal^\cyl(W')$ is Fr\'echet differentiable in $W'$, with the derivative given by
\[ \Psi'(u) = \sum_{j=1}^m \partial_j\psi(\average{u,w_1}_{W',W}, \ldots, \average{u,w_m}_{W', W})w_j \in W,
\]
where $\partial_j\psi$ is the partial derivative of $\psi$ with respect to the $j$-th variable. Notice that $\Psi'(u)$ belongs to $W$, and the map $u\mapsto \Psi'(u)$ is continuous from $W'$ into $W$.

If $W$ is dense in $X$ and $K$ is a compact set in $W'$, then it follows from the Stone-Weierstrass Theorem that the set $\left.\Tcal^\cyl(W')\right|_K$, obtained as the space of functionals which are the restriction to $K$ of a functional in $\Tcal^\cyl(W')$, is dense in the space $\Ccal(K)$. The density of the space $W$ in $X$ is needed in order to have that the algebra $\left.\Tcal^\cyl(W')\right|_K$ separates points, which is one of the requirements of the Stone-Weierstrass Theorem.

The set $\Tcal^\cyl(W')$ is a normed vector space with the norm
\begin{equation}
  \label{onenormforPhi}
  \|\Psi\|_1 = \sup_{u \in W'} \left(|\Psi(u)| + \|\Psi'(u)\|_W\right).
\end{equation}
As mentioned above, $\Psi'(u)$ belongs to $W$, so the definition above makes sense.

With the definition of cylindrical test functions at hand, we introduce the notion of weak stationary statistical solution.
\begin{defs}
  \label{defwsss}
  A \textbf{weak stationary statistical solution} of equation \eqref{diffeq} is a Borel probability measure $\mu$ on the space $X$ such that, for any cylindrical test function $\Psi\in \Tcal^\cyl(W')$, the map $u\mapsto \dual{F(u),\Psi'(u)}_{W', W}$ is $\mu$-integrable and
  \begin{equation}
    \label{ssseq}
    \int_X \dual{F(v),\Psi'(v)}_{W', W} \;\rmd \mu(v) = 0.
  \end{equation}
  Given a Borel subset $E$, we denote by $\Pcal_\wsss(E)$ the set of all weak stationary statistical solutions which are carried by $E$, i.e. $\mu(X\setminus E) = 0$.
\end{defs}

The term \emph{weak} in the definition above is akin to that of weak solutions of differential equations. In specific problems, further conditions can be added, giving rise to a subset of weak stationary statistical solutions. This is the case in the three-dimensional Navier-Stokes equations, in which the stationary statistical solutions are carried by a more regular space and they enjoy a mean energy-type inequality, leading to the definition of Foias-Prodi stationary statistical solution \cite{Foias73, FMRT2001, FRT2019}, akin to that of Leray-Hopf weak solution. See Section \ref{secweakandfpsss} for more details in the case of the Navier-Stokes equations.

\subsection{Continuous semigroups}
\label{subsecsemigroups}

We recall the definition of a continuous semigroup.
\begin{defs}
  A \textbf{continuous semigroup} on a Hausdorff topological space $X$ is a family $\{S(t)\}_{t\geq 0}$ of operators in $X$ satisfying
\begin{enumerate}
  \item $(t,u) \mapsto S(t)u$ is a continuous map from $[0,\infty)\times X$ into $X$;
  \item $S(0)$ is the identity in $X$;
  \item $S(t+s) = S(t)S(s)$, for all $t,s\geq 0$;
\end{enumerate}
  We say that the continuous semigroup is \textbf{associated with the equation} \eqref{diffeq} when each solution $u(t) = S(t)u_0$, $t\geq 0$, $u_0\in X$, solves the equation \eqref{diffeq} in $W'$. This means that, for each $w\in W$, $t\rightarrow \average{u(t),w}_{W',W}$ is continuously differentiable from $[0,\infty)$ into $\RR$, and
\[ \frac{\rmd}{\rmd t} \average{u(t),w}_{W',W} = \average{F(u(t)), w}_{W',W},
\]
for all $t\geq 0$.
\end{defs}

A set $B\subset X$ is called \textbf{positively invariant} for the semigroup when $S(t)u_0 \in B$, for all $t\geq 0$ and all $u_0\in B$.

Associated with a continuous semigroup, we recall the definition of an invariant measure.
\begin{defs}
  \label{definvmeas}
  An \textbf{invariant measure} for a continuous semigroup $\{S(t)\}_{t\geq 0}$ on a Hausdorff topological space $X$ is a Borel probability measure $\mu$ satisfying
  \[ \mu(S(t)^{-1}E) =\mu(E),
  \]
  for all Borel subsets of $X$ and all $t\geq 0$. Given a Borel subset $E$, we denote by $\Pcal_\inv(E)$ the set of all invariant measures which are carried by $E$, i.e. $\mu(X\setminus E) = 0$.
\end{defs}

Two other important notions are those of pointwise absorbing and attracting sets:
\begin{defs}
  \label{defabsattr}
  Let  $\{S(t)\}_{t\geq 0}$ be a continuous semigroup on a Hausdorff topological space $X$. Let $A$ and $B$ be two subsets of $X$. We say that $A$ \textbf{absorbs the points of} $B$ if, for each $u_0\in B$, there exists $T\geq 0$ such that $S(t)u_0\in A$, for all $t\geq T$. We say that $A$ \textbf{attracts the points of} $B$ if, for each $u_0\in B$, $S(t)u_0$ converges to $A$ in the sense that, for every open neighborhood $O$ of $A$, there exists $T>0$ such that $u(t)\in O$, for all $t\geq T$.
\end{defs}
Notice this is different from the usual notions of absorbing set and global attractor where the absorption and the attracting properties are \emph{uniform} with respect to the points in $B$ (see e.g. \cite{BabinVishik1992, SellYou2002, Temam1988}).

\section{Mathematical background}

Here, we recall a few classical results in Functional Analysis, Measure Theory and Ergodic Theory, based on \cite{AliBor2006, CoFoSi1982, DunSch1958, EkeTem1976, Komiya1988, Krengel1985, Rudin1987, Simon2011, Sion1958, Walters1982} and other references cited below.

\subsection{Generalized limits}
\label{secgenlim}

Let $\Bcal([0,\infty))$ be the set of all bounded real-valued functions on $[0,\infty)$. A \textbf{generalized limit}, denoted $\Lim$, is any positive linear functional on $\Bcal([0,\infty))$ that extends the usual limit. It is characterized by the following properties:
\begin{enumerate}
  \item $\Lim$ is a linear map from $\Bcal([0,\infty))$ into $\RR$;
  \item $\Lim g = \lim_{T\rightarrow \infty} g(T)$, for every $g\in\Bcal([0,\infty))$ such that the classical limit exists;
  \item $\Lim g \geq 0$, for all $g\in \Bcal([0,\infty))$ with $g\geq 0$.
\end{enumerate}

The existence of generalized limits is guaranteed by the Hahn-Banach Theorem, by extending the classical limit to the whole space $\Bcal([0,\infty))$, using $p(g) = \limsup g$ as a gauge function (see e.g. \cite{DunSch1958, FMRT2001}).

Any such generalized limit has the property of being invariant for time-translations of time averages of functions in $L^\infty(0,\infty)$, i.e.
\begin{equation}
  \label{timeaveinvgenlim}
  \Lim \frac{1}{T}\int_0^T f(t+\tau)\;\rmd t = \Lim \frac{1}{T}\int_0^T f(t)\;\rmd t, 
\end{equation}
for all $\tau \geq 0$, for any $f\in L^\infty(0,\infty)$ (see \cite[Appendix IV.A.2]{FMRT2001}).

From the condition (iii), it follows that
\[ \liminf_{T\rightarrow \infty} g(T) \leq \Lim g \leq \limsup_{T\rightarrow \infty} g(T),
\]
for any $g\in\Bcal([0,\infty))$.

Moreover, given any $g_0$ for which the usual limit does not exist and given any sequence $t_n\rightarrow \infty$ for which $g_0(t_n)$ converges to a certain value $\ell_0$ (e.g., the limsup or liminf), we can enforce that $\Lim g_0 = \ell_0$ (see \cite[Remark IV.1.5]{FMRT2001}).

\subsection{The topology of the space of probability measures}
\label{secprobspaces}

Given a topological space $X$, we consider the space $\Pcal(X)$ of Borel probability measures on $X$ endowed with the standard weak-star topology, which is the smallest topology for which the mappings $u \mapsto \int_X \varphi(u)\;\rmd\mu(u)$ are continuous, for all $\varphi$ in the space $\Ccal_b(X)$ of continuous and bounded real-valued functions on $X$. If a net $(\mu_\alpha)_\alpha$ converges to $\mu$ in this topology, we write $\mu_\alpha \rightharpoonup \mu$, bearing in mind that convergence in this topology means that $\int_X \varphi(u)\;\rmd\mu_\alpha(u)\rightarrow \int_X \varphi(u)\;\rmd\mu(u)$, for every such $\varphi$.

In such a generality, $\Pcal(X)$ is not well behaved. In particular, even if $X$ is a regular space,  $\Pcal(X)$ may not be a Hausdorff space, in the sense that we may have two different measures $\mu_1$ and $\mu_2$ for which $\int_X \varphi(u)\;\rmd\mu_1(u) = \int_X \varphi(u)\;\rmd\mu_2(u)$ for every $\varphi\in\Ccal_b(X)$. If, however, $K$ is a metrizable, or compact, subspace of $X$, then $\Pcal(K)$ is a Hausdorff space \cite[Section 15.1]{AliBor2006}.

Concerning compactness, if $K$ is a metrizable space, it follows from \cite[Theorem 15.11]{AliBor2006} that $\Pcal(K)$ is compact if and only if $K$ is a compact (see also the related Prohorov Theorem \cite{Prohorov1956}).

\subsection{Time-average measures}
\label{sectimeave}

Let $B$ be a closed subset of a Hausdorff topological space $X$ and assume $B$ is a normal topological space (i.e. any two disjoint closed sets can be separated by disjoint open neighborhoods) with the topology inherited from $X$. Let $u:[0,\infty)\rightarrow B$ be a continuous mapping and suppose there exists a compact subset $K$ of $X$ that attracts $u$ in the same sense as in Definition \ref{defabsattr}, i.e. for every open neighborhood $O$ of $K$, there exists $T>0$ such that $u(t)\in O$, for all $t\geq T$.

For any $\varphi\in \Ccal_b(B)$, the boundedness of $\varphi$ implies the boundedness of the time averages
\[ T \mapsto \frac{1}{T} \int_0^T \varphi(u(t))\;\rmd t.
\]

From this boundedness and the discussions in Section \ref{secgenlim}, it follows that the generalized limit of these time averages is well defined:
\[ \Lim \frac{1}{T} \int_0^T \varphi(u(t))\;\rmd t.
\]

Since $B$ is closed and $K$ is compact, the set $K\cap B$ is a compact set. Given $\varphi\in \Ccal(K\cap B)$, since $K\cap B$ is compact and $B$ is normal, it follows from Tietze's Extension Theorem \cite[Corollary I.5.4]{DunSch1958} that there exists a continuous and bounded function $\tilde\varphi\in \Ccal_b(B)$ such that $\tilde\varphi(u) = \varphi(u)$, for all $u\in K\cap B$. If $\varphi_1$ and $\varphi_2$ are two such extensions, we have $\varphi_1 - \varphi_2 = 0$ on $K\cap B$. Then, by the compactness of $K\cap B$ and the continuity of $\varphi_1 - \varphi_2$, given any $\varepsilon > 0$, there exists an open neighborhood $O\supset K\cap B$ such that
\[ |\varphi_1(u) - \varphi_2(u)| < \varepsilon, \qquad \forall u \in O.
\]
Since $K$ attracts $u$, there exists $T_0>0$ such that $u(t) \in O$, for all $t\geq T_0$. This implies that, using the linearity of the generalized limit,
\begin{multline*}
   \left|\Lim \frac{1}{T} \int_0^T \varphi_1(u(t))\;\rmd t - \Lim \frac{1}{T} \int_0^T \varphi_2(u(t))\;\rmd t\right| \\
    \leq \limsup_{T\rightarrow \infty} \frac{1}{T} \int_0^T |\varphi_1(u(t)) - \varphi_2(u(t))| \;\rmd t<\varepsilon.
\end{multline*}
Since $\varepsilon > 0$ is arbitrary, this means that the generalized limit is independent of the extension. This allows us to define a continuous linear function $\Lambda : \Ccal(K\cap B) \rightarrow \mathbb{R}$ by
\[ \Lambda(\varphi) = \Lim \frac{1}{T} \int_0^T \tilde\varphi(u(t))\;\rmd t,
\]
where $\tilde\varphi\in \Ccal_b(B)$ is any extension of $\varphi$ to $\Ccal_b(B)$.

If $\varphi\geq 0$ on $K\cap B$, Tietze's Extension Theorem guarantees that there is an extension $\tilde\varphi$ of $\varphi$ that is also nonnegative, i.e. $\tilde\varphi\geq 0$ on $B$. Thus, $\Lambda$ is a positive continuous linear function on $\Ccal(K\cap B)$. Hence, by the Kakutani-Riesz Representation Theorem \cite{Rudin1987}, it is represented by a Borel probability measure on $K\cap B$, which we denote by $\mu$. Thus,
\[ \Lim \frac{1}{T} \int_0^T \tilde\varphi(u(t))\;\rmd t = \Lambda(\varphi) = \int_{K\cap B} \varphi(v) \;\rmd\mu(v),
\]
for all $\varphi\in \Ccal(K\cap B)$ and for any extension $\tilde\varphi \in \Ccal_b(B)$.

On a different perspective, if $\varphi\in \Ccal_b(B)$ is given, its restriction $\varphi|_{K\cap B}$ belongs to $\Ccal(K\cap B)$, so that
\[ \Lim \frac{1}{T} \int_0^T \varphi(u(t))\;\rmd t = \Lambda(\varphi|_{K\cap B}) = \int_{K\cap B} \varphi(v) \;\rmd\mu(v).
\]

Extending $\mu$ from $K\cap B$ to $X$ in the usual way (i.e. $\mu(E) = \mu(E\cap (K\cap B))$, for all Borel sets $E$ in $X$), we find that $\mu$ is a Borel probability measure in $X$ which is carried by $K\cap B$ and is such that
\begin{equation}
  \label{genlimandmu}
  \Lim \frac{1}{T} \int_0^T \varphi(u(t))\;\rmd t = \int_X \varphi(v) \;\rmd\mu(v), \quad \forall \varphi\in \Ccal_b(B).
\end{equation}

Given a particular $\phi\in \Ccal_b(B)$ and a particular sequence $T_n\rightarrow \infty$ such that the time averages of $\phi$ over $[0,T_n]$ converge to a limit, such as the limsup or the liminf, then, just as in Section \ref{secgenlim}, we may choose the generalized limit that yields the limit of the averages with $T_n\rightarrow \infty$. We are particularly interested in the limsup. Hence, given a particular $\phi$, we choose a generalized limit that yields the limsup, and with the associated measure still denoted by $\mu$. In this case, we have
\begin{equation}
  \label{genlimparticularphi}
  \int_X \phi(v) \;\rmd\mu(v) = \Lim \frac{1}{T} \int_0^T \phi(u(t))\;\rmd t = \limsup_{T\rightarrow \infty} \frac{1}{T} \int_0^T \phi(u(t))\;\rmd t,
\end{equation}
for this particular $\phi$, with
\begin{multline}
  \label{genlimarbitraryvarphi}
  \liminf_{T\rightarrow \infty} \frac{1}{T} \int_0^T \varphi(u(t))\;\rmd t \leq \int_K \varphi(v) \;\rmd\mu(v) = \Lim \frac{1}{T} \int_0^T \varphi(u(t))\;\rmd t \\
   \leq \limsup_{T\rightarrow \infty} \frac{1}{T} \int_0^T \varphi(u(t))\;\rmd t,
\end{multline}
for an arbitrary $\varphi\in \Ccal_b(B)$.

\begin{rmk}
  We use the two notations $\phi$ and $\varphi$ to distinguish a given particular $\phi$ for which \eqref{genlimparticularphi} holds and an arbitrary $\varphi$ for which \eqref{genlimandmu} and \eqref{genlimarbitraryvarphi} hold but \eqref{genlimparticularphi} may not hold. In our main results, $\phi$ is a given function for which we want to bound the time or ensemble averages associated with the system, while $\varphi$ is an auxiliary function used to define mathematical objects (such as the operator $\Lambda$ above) and prove some of our results.
\end{rmk}

We summarize the results of this section as follows.

\begin{lem}
  \label{kattractingtimeavemeas}
  Let $X$ be a Hausdorff topological space and let $B$ be a closed subset of $X$. Assume $B$ is a normal topological space with the topology inherited from $X$. Let $u:[0,\infty)\rightarrow B$ be a continuous mapping and suppose there exists a compact subset $K$ of $X$ that attracts $u$. Finally, let $\Lim$ be a generalized limit. Then, there exists a Borel measure $\mu$ on $X$ which is carried by $K\cap B$ and satisfies \eqref{genlimandmu} and \eqref{genlimarbitraryvarphi}. Moreover, given $\phi\in\Ccal_b(B)$, both the generalized limit and the measure $\mu$ can be chosen to also satisfy \eqref{genlimparticularphi}.
\end{lem}

\subsection{Extreme points and Dirac measures}
\label{secextremeanddirac}

If $Y$ is a separable and metrizable topological space, it follows from \cite[Theorem 15.9]{AliBor2006} that the set of extreme points of $\Pcal(Y)$ is the set of Dirac probability measures $\delta_u$, with $u\in Y$. Since any compact and metrizable space is separable, this holds, in particular, for a compact and metrizable space $K$.

This result is usually used in combination with minimization or maximization problems. In this case, the Bauer Maximum Principle says that if $K$ is a compact convex subset of a locally convex Hausdorff space, then every upper semicontinuous convex function on $K$ has a maximum that is an extreme point (see \cite[7.69 Bauer Maximum Principle]{AliBor2006}).

\subsection{Ergodic theorem}

We state below a version of the Birkhoff-Khinchin Ergodic Theorem as it applies to our context. Here, we consider the space $\Pcal_\inv(X)$ of invariant measures given in Definition \ref{definvmeas}.
\begin{thm}[Birkhoff-Khinchin Ergodic Theorem]
  \label{thmbirkhiergodic}
  Consider a Hausdorff topological space $X$ and a continuous semigroup $\{S(t)\}_{t\geq 0}$ on $X$. 
  Let $\mu\in \Pcal_\inv(X)$ and let $\varphi\in L^1(\mu)$. Then, for $\mu$-almost every $u\in X$, the limit
  \[ \lim_{T\rightarrow \infty} \frac{1}{T}\int_0^T \varphi(S(t) u)\;\rmd t
  \] 
  exists.
\end{thm}

For a proof of the Birkhoff-Khinchin Ergodic Theorem in the context of semiflows, see \cite[Ch 1, \S 2, Theorem 1]{CoFoSi1982} or \cite[Section 1.2]{Krengel1985}, or see our earlier work \cite[Section 3]{FRT2015}, where the result is adapted from the version for semigroups given in \cite[Theorem VIII.7.5]{DunSch1958}.

Another important result is that the extreme points of invariant measures are ergodic. This result can be found for discrete maps or continuous groups in several works, e.g. \cite[Theorem 9.13]{Walters1982}, \cite[Chapter 8]{Simon2011}, \cite[Theorem 19.25]{AliBor2006}. The proof is not complicated and can be easily adapted to our context.

Indeed, consider the case where $K$ is compact and metrizable. In this case, $\Pcal(K)$ is also compact and metrizable. Since the semigroup is continuous, it is straightforward to deduce that the subset $\Pcal_\inv(K)$ of invariant probability measures is closed in $\Pcal(K)$, hence compact. Now, $\mu\in \Pcal_\inv(K)$ is not an extreme point in $\Pcal_\inv(K)$ if and only if $\mu$ can be decomposed as $\mu = \theta \mu_1 + (1-\theta)\mu_2$, with $\mu_1, \mu_2\in \Pcal_\inv(K)$ and $0<\theta <1$, which is equivalent to saying that $\mu$ is not ergodic. In other words, $\mu$ is ergodic if and only if it is an extreme point in $\Pcal_\inv(K)$. 

In case $f$ is an upper semicontinuous convex function on $\Pcal_\inv(K)$, then it follows from the Bauer Maximum Principle (see Section \ref{secextremeanddirac}) that there exists a maximum $\mu^*$ of $f$ which is an extreme point of $\Pcal_\inv(K)$ and, hence, it is ergodic.

With $\mu^*$ ergodic, and using the Birkhoff Ergodic Theorem \cite[Theorem 1.14]{Walters1982} for ergodic measures, we find that the limit in Theorem \ref{thmbirkhiergodic} is constant and is equal to the mean value of $\varphi$ with respect to the measure $\mu^*$.

Thus, we have obtained the following classical results.
\begin{thm}
  \label{thmextremeinvariantmeasareergodic}
  Let $\{S(t)\}_{t\geq 0}$ be a continuous semigroup on a Hausdorff topological space $X$. Suppose $K$ is a compact and metrizable subspace of $X$. Then, $\Pcal_\inv(K)$ is compact, and a measure $\mu$ in $\Pcal_\inv(K)$ is ergodic if and only if $\mu$ is an extreme point in $\Pcal_\inv(K)$. Moreover, if $f$ is an upper semicontinuous convex function on $\Pcal_\inv(K)$, then it has a maximum at some $\mu^*$ which is an extreme point in $\Pcal_\inv(K)$ and, hence, it is ergodic. Being ergodic, we have that
  \[ \lim_{T\rightarrow \infty} \frac{1}{T}\int_0^T \varphi(S(t) u)\;\rmd t = \int_K \varphi(u)\;\rmd\mu^*(u),
  \]
  for $\mu^*$-almost every $u\in K$.
\end{thm}

\subsection{Invariant measures and weak stationary statistical solutions}

In the context of Section \ref{subsecsemigroups}, an important ingredient for the optimality of the minimax estimate is that the notion of weak stationary statistical solution be equivalent to that of invariant measure. This, however, is a delicate issue, and it is not known whether it is true in general.

One direction is actually straightforward, namely, that any invariant measure is a weak stationary statistical solution provided $F$ is bounded on a carrier of the measure. Depending on the equation, this condition on $F$ may be relaxed (see Remark \ref{rmkrelaxcompctforinvissss}).

The converse, however, is not known to be true in such a generality. It has been proved, though, for the two-dimensional Navier-Stokes equations in \cite[Proposition 6.2]{Foias73} and for a globally modified three-dimensional Navier-Stokes equations obtained by truncating the nonlinear term in \cite[Theorem 15]{KloMRRea2009}. The proof relies on approximating the semigroup by a differentiable and left-invertible semigroup, and it depends also on involved estimates. As such, it can indeed be showed to hold for a large number of equations \cite{RTinpreparation}.

Below, we prove the direction that is valid in general, leaving the converse to each application.
\begin{prop}
  \label{propinvimplieswsss}
  In the framework of Sections \ref{subsecsss} and \ref{subsecsemigroups}, given a Borel subset $E$ of $X$ and assuming $F$ is bounded on $E$, it follows that $\Pcal_\inv(E)\subset\Pcal_\wsss(E)$.
\end{prop}

\begin{proof}
  Let $\Psi\in \Tcal^\cyl(W')$ and let $\psi$, $m$, and $w_1, \ldots, w_m$ be as in \eqref{defscyl}. For each $u\in X$, the orbit $t\mapsto S(t)u$, $t\geq 0$, is continuously differentiable as a function in $W'$. Thus, for each $w_j$, $j=1, \ldots, m$, the map $t\mapsto \dual{S(t)u, w_j}_{W', W}$ is continuously differentiable. Hence, the composition $t\mapsto \Psi(S(t)u)$ is continuously differentiable, with
  \begin{multline*}
    \frac{\rmd}{\rmd t} \Psi(S(t)u) = \frac{\rmd}{\rmd t} \psi(\dual{S(t)u,w_1}_{W',W}, \ldots, \dual{S(t)u,w_m}_{W', W}) \\
    = \sum_{j=1}^m \partial_j\psi(\dual{S(t)u,w_1}_{W',W}, \ldots, \dual{S(t)u,w_m}_{W', W})\frac{\rmd}{\rmd t}\dual{S(t)u,w_j}_{W', W} \\
    = \sum_{j=1}^m \partial_j\psi(\dual{S(t)u,w_1}_{W',W}, \ldots, \dual{S(t)u,w_m}_{W', W})\dual{F(S(t)u),w_j}_{W', W} \\    
    = \dual{F(S(t)u),\Psi'(S(t)u)}_{W',W}.
  \end{multline*}
  Integration in time from $0$ to $t$ yields
  \[ \Psi(S(t)u) = \Psi(u) + \int_0^t \dual{F(S(\tau)u),\Psi'(S(\tau)u)}_{W',W}\;\rmd \tau.
  \]
  Integration in $u$ over $K$ yields
  \begin{multline*}
    \int_K \Psi(S(t)u)\;\rmd\mu(u) = \int_K \Psi(S(t)u)\;\rmd\mu(u) \\
      + \int_K \int_0^t \dual{F(S(\tau)u),\Psi'(S(\tau)u)}_{W',W}\;\rmd \tau \;\rmd\mu(u).
  \end{multline*} 
  By the continuity of the semigroup,  the continuity of $\Psi'$, and the assumption that $F$ is a Borel map, it follows that the integrand above is Borel on $[0, t]\times K$, hence $\lambda \times \mu$-measurable, where $\lambda$ denotes the Lebesgue measure on $\RR$. Since $\Psi'$ is bounded on $W$ and $F$ is assumed to be bounded on $W'$, then the integrand is $\lambda\times\mu$-integrable; in fact it is in $L^\infty(\lambda\times\mu)$. Thus, we apply the Fubini Theorem and obtain
  \begin{multline*}  
  \int_K \Psi(S(t+s)u)\;\rmd\mu(u) = \int_K \Psi(S(t)u)\;\rmd\mu(u) \\
    = \int_K \Psi(S(t)u)\;\rmd\mu(u) + \int_0^t \int_K \dual{F(S(\tau)u),\Psi'(S(\tau)u)}_{W',W} \;\rmd\mu(u)\;\rmd \tau.
  \end{multline*}
  This shows that
  \[ t \mapsto \int_K \Psi(S(t)u)\;\rmd\mu(u)
  \]
  is continuously differentiable, with
  \[ \frac{\rmd}{\rmd t} \int_K \Psi(S(t)u)\;\rmd\mu(u) = \int_K \dual{F(S(t)u),\Psi'(S(t)u)}_{W',W} \;\rmd\mu(u).
  \]
  Now, using that $\mu$ is invariant, the left hand side above vanishes and we are left with
  \[  \int_K \dual{F(S(t)u),\Psi'(S(t)u)}_{W',W} \;\rmd\mu(u) = 0,
  \]
  for any $t\geq 0$. In particular, for $t=0$, we see that $\mu$ is a weak stationary statistical solution.
\end{proof}

\begin{rmk}
  \label{rmkrelaxcompctforinvissss}
  The boundedness of $F$ in Proposition \ref{propinvimplieswsss} was used only to assure that the function $(\tau, u) \mapsto \dual{F(S(\tau)u),\Psi'(S(\tau)u)}_{W',W}$ is bounded, hence integrable, and, thus, Fubini's Theorem can be used. In applications, using appropriate a~priori estimates for the solutions, this condition can be relaxed to any condition guaranteeing that this function is integrable on $[0,t]\times E$. For instance, the measure may be carried by a more regular space, with a finite moment (such as mean enstrophy as in the case of the Navier-Stokes equations), and with $F$ controlled by this moment.
\end{rmk}

\subsection{Minimax Theorem}
An important step in the proof of the optimal bound is a minimax theorem that allows us to switch the order of the maximization over all weak stationary statistical solutions and minimization over all cylindrical test functions. As in \cite{TobGolDoe2017}, we use the minimax result from Sion \cite{Sion1958}, as stated in the Introduction of \cite{Komiya1988}.

\begin{thm}[Sion's Minimax Theorem]
  \label{thmsionminimax}
  Let $U$ be a compact convex subset of a topological vector space and $Z$ a convex subset of a possibly different topological vector space. Let $f:U\times Z \rightarrow \RR$ be such that
  \begin{enumerate}
    \item $f(u,\cdot)$ is upper semicontinuous and quasi-concave on $Z$, for each $u\in U$;
    \item $f(\cdot,v)$ is lower semicontinuous and quasi-convex on $U$, for each $v\in Z$.
  \end{enumerate}
  Then,
  \[ \max_{u\in  U} \inf_{v\in Z} f(u,v) = \inf_{v\in Z} \max_{u\in U} f(u,v).
  \]
\end{thm}

Recall that, given two sets $A$ and $B$, a function $f:A\times B\rightarrow\RR$ is called \textbf{quasi-concave} in $A$ if $\{u\in A; \; f(u,v)\geq c\}$ is a convex set in $A$, for any $v\in B$ and any $c\in \RR$, and is called \textbf{quasi-convex} in $B$ if $\{v\in B; \; f(u,v)\leq c\}$ is a convex set in $B$, for any $u\in A$ and any $c\in \RR$.

We actually apply this result to functions which are linear in one variable and affine in the other, so they are quasi-concave and quasi-convex on either variable. Recall that, if $X$ is a vector space and $A\subset X$ is convex, then a function $f:A\rightarrow \RR$ is called \textbf{affine} when $f(\theta u +(1-\theta)v) = \theta f(u) + (1-\theta) f(v)$, for all $u, v\in A$ and all $0\leq \theta \leq 1$.

If one is willing to give up the optimality result but still retain the minimax upper bound estimate, the following result is useful (see \cite[Proposition VI.1.1]{EkeTem1976}).

\begin{prop}
  \label{propminimaxupperbound}
  Let $A$ and $B$ be two arbitrary sets and consider $f:A\times B \rightarrow \RR$ arbitrary. Then, 
  \[  \sup_a \inf_b f(a,b) \leq \inf_b \sup_a f(a,b).
  \]
\end{prop}

\section{Optimal minimax estimate for weak stationary statistical solutions}

In this section, we assume the framework and conditions described in Section \ref{subsecsss}, namely, that $X$ is a Hausdorff space, $W$ is a Banach space, the inclusion $X\subset W'$ is continuous, and $F:X\rightarrow W'$ is a Borel map. In this context, we have the following main result.
\begin{thm}
  \label{thmminimaxforsss}
  Suppose $K$ is a compact and metrizable subset of $X$ and $F$ is continuous on $K$. Assume the set $\Pcal_\wsss(K)$ of weak stationary statistical solutions carried by $K$ is not empty. Let $\phi\in \Ccal(K)$. Then,
  \begin{equation}
  \label{minimaxforssseq}
    \max_{\mu\in\Pcal_\wsss(K)} \int_K \phi(u) \;\rmd\mu(u) = \inf_{\Psi\in \Tcal^\cyl(W')} \max_{u\in K} \left\{\phi(u) + \dual{F(u), \Psi'(u)}_{W', W}\right\}.
  \end{equation}
\end{thm}

\begin{proof}
  We prove the series of steps
  \begin{subequations}
  \begin{align}
    & \label{maxsss} \max_{\mu\in\Pcal_\wsss(K)} \int_K \phi(u) \;\rmd\mu(u) \\
    & \label{maxsss_to_maxsssinf} \quad = \max_{\mu\in\Pcal_\wsss(K)} \inf_{\Psi\in\Tcal^\cyl(W')} \int_K \left\{\phi(u) + \dual{F(u), \Psi'(u)}_{W', W}\right\} \;\rmd\mu(u) \\
    & \label{maxsssinf_to_maxmeasinf} \quad = \max_{\mu\in\Pcal(K)}\inf_{\Psi\in\Tcal^\cyl(W')} \int_K \left\{\phi(u) + \dual{F(u), \Psi'(u)}_{W', W}\right\} \;\rmd\mu(u) \\
    & \label{maxmeasinf_to_infmaxmeas} \quad = \inf_{\Psi\in\Tcal^\cyl(W')} \max_{\mu\in\Pcal(K)} \int_K \left\{\phi(u) + \dual{F(u), \Psi'(u)}_{W', W}\right\} \;\rmd\mu(u) \\
    & \label{infmaxmeas_to_infmaxpointwise} \quad = \inf_{\Psi\in \Tcal^\cyl(W')} \max_{u\in K} \left\{\phi(u) + \dual{F(u), \Psi'(u)}_{W', W}\right\}.
  \end{align}
\end{subequations}

\smallskip
\noindent \textbf{Step \eqref{maxsss} is well defined:}

Let $\phi$ be as in the statement of the theorem. Since $\phi$ is continuous on $K$ and $K$ is compact, then the integrand is $\mu$-measurable and the integral is well-defined and bounded. Moreover, the integral depends continuously on $\mu$ in the topology of $\Pcal_\wsss(K)$. Since, again, $K$ is compact and metrizable, if follows from \cite[Theorem 15.11]{AliBor2006} that $\Pcal(K)$ is compact. Since $\Pcal_\wsss(K)$ is closed in $\Pcal(K)$, then $\Pcal_\wsss(K)$ is also compact. Thus, the maximum in \eqref{maxsss} is achieved, showing that \eqref{maxsss} is well-defined.

\smallskip
\noindent \textbf{Step \eqref{maxsss} equals \eqref{maxsss_to_maxsssinf}:}

This step follows directly from the definition of weak stationary statistical solutions as satisfying the condition \eqref{ssseq}. In fact, condition \eqref{ssseq} implies directly that, for any $\mu\in \Pcal_\wsss(K)$ and any $\Psi\in \Tcal^\cyl(W')$, we have
\[   \int_K \phi(u) \;\rmd\mu(u) = \int_K \left\{\phi(u) + \dual{F(u), \Psi'(u)}_{W', W}\right\} \;\rmd\mu(u).
\]
Taking first the infimum with respect to $\Psi$ and then the maximum with respect to $\mu$ (which exists as proved in the previous step) leads to the identity between \eqref{maxsss} and \eqref{maxsss_to_maxsssinf}.

\smallskip
\noindent \textbf{Step \eqref{maxsss_to_maxsssinf} equals \eqref{maxsssinf_to_maxmeasinf}:}

If a measure $\mu$ in $\Pcal(K)$ is not a weak stationary statistical solution, then \eqref{ssseq} does not hold. Hence, for any $\mu\in \Pcal(K) \setminus \Pcal_\wsss(K)$, there exists $\Psi_0\in \Tcal^\cyl(W')$ such that
\[ \int_K \dual{F(\bfu), \Psi_0'(\bfu)}_{W',W} \;\rmd\mu(\bfu) \neq 0.
\]
Notice that, for $\Psi_0\in \Tcal^\cyl(W')$, we have  $\bfu\mapsto\Psi_0'(\bfu)$ continuous from $W'$ into $W$ and $\bfu\mapsto F(\bfu)$ continuous from $K$ into $W'$, so that $\bfu \mapsto \dual{F(\bfu), \Psi_0'(\bfu)}_{W',W}$ is continuous on $K\subset X \subset W'$, hence $\mu$-integrable.

Considering the family $\{\lambda\Psi_0\}_{\lambda\in\RR} \subset \Tcal^\cyl(W')$, we have 
\begin{multline*}
  \inf_{\Psi\in\Tcal^\cyl(W')} \int_K \left\{\phi(u) + \dual{F(u), \Psi'(u)}_{W', W}\right\} \;\rmd\mu(u) \\ \leq \inf_{\lambda\in \RR} \int_K \left\{\phi(u) + \lambda \dual{F(u), \Psi_0'(u)}_{W', W}\right\} \;\rmd\mu(u)  = -\infty.
\end{multline*}
Thus,
\[  \max_{\mu\in\Pcal(K)\setminus \Pcal_\wsss(K)} \inf_{\Psi\in\Tcal^\cyl(W')} \int_K \left\{\phi(u) + \dual{F(u), \Psi'(u)}_{W', W}\right\} \;\rmd\mu(u) = -\infty.
\]
This means that the supremum for $\mu$ in $\Pcal(K)$ must be achieved on $\Pcal_\wsss(K)$, and is thus a maximum, proving the equality between \eqref{maxsss_to_maxsssinf} and\eqref{maxsssinf_to_maxmeasinf}.

\smallskip
\noindent \textbf{Step \eqref{maxsssinf_to_maxmeasinf} equals \eqref{maxmeasinf_to_infmaxmeas}:}

Here, we apply the Minimax Theorem \ref{thmsionminimax}, with $U=\Pcal(K)$, $Z=\Tcal^\cyl(W')$ and
\begin{equation}
  \label{phiplusdualitymap}
  f(\mu, \Psi) = \int_K \left\{\phi(u) + \dual{F(u), \Psi'(u)}_{W', W}\right\} \;\rmd\mu(u).
\end{equation}
Let us check all the necessary conditions.

As seen in Section \ref{secprobspaces}, since $K$ is compact and metrizable, it follows from \cite[Theorem 15.11]{AliBor2006} that $\Pcal(K)$ is compact. As any probability space, $\Pcal(K)$ is also convex. Hence, $\Pcal(K)$ is a compact convex subspace of the Banach space $\Ccal(K)'$. On the other hand, the space $\Tcal^\cyl(W')$ is convex and itself is a normed vector space under the norm \eqref{onenormforPhi}. 

Furthermore, the function $f$ is linear with respect to $\mu$ in $\Pcal(K)$ and is affine with respect to $\Psi$ in $\Tcal^\cyl(W')$. Thus, it is trivially quasi-concave on $\Tcal^\cyl(W')$, for each $\mu$ in $\Pcal(K)$, and quasi-convex in $\mu$, for each $\Psi$ in $\Tcal^\cyl(W')$.

Finally, we show that $f$ is continuous on $\Pcal(K)\times\Tcal^\cyl(W')$. For that purpose, we use \cite[Corollary 15.7]{AliBor2006}, which says, since $K$ is metrizable, that  $(\mu, g)\mapsto \int_K g\;\rmd \mu$ is continuous on $\Pcal(K)\times\Ccal_b(K)$. With that in mind, we let $g_\Psi(u) = \phi(u) + \dual{F(u), \Psi'(u)}_{W', W}$, so that $f(\mu,\Psi) = \int_K g_\Psi\;\rmd\mu$. Thus, it suffices to show that the map taking $\Psi$ into $g_\Psi$ is continuous from $\Tcal^\cyl(W')$ into $\Ccal_b(K)$.

Notice that $\Psi\mapsto\Psi'$ is continuous from $\Tcal^\cyl(W')$ into $\Ccal_b(K,W)$. Moreover, $F$ is continuous from $K$ into $W'$. Thus, the map taking $\Psi$ into $u\mapsto \dual{F(u), \Psi'(u)}_{W', W}$ is continuous on $\Tcal^\cyl(W')$. Adding the function $\phi$, which is continuous on $K$, shows that $\Psi\mapsto g_\Psi$ is continuous from $\Tcal^\cyl(W')$ into $\Ccal_b(K)$.

Thus, all the conditions in Theorem \ref{thmsionminimax} are met, and it implies the equality between \eqref{maxsssinf_to_maxmeasinf} and \eqref{maxmeasinf_to_infmaxmeas}.

\smallskip
\noindent \textbf{Step \eqref{maxmeasinf_to_infmaxmeas} equals \eqref{infmaxmeas_to_infmaxpointwise}:}

It suffices to prove that
\[  \max_{\mu\in \Pcal(K)} \int_K \left\{\phi(u) + \dual{F(u), \Psi'(u)}_{W', W}\right\} \;\rmd\mu(u)
   = \max_{u\in K} \left\{\phi(u) + \dual{F(u), \Psi'(u)}_{W', W}\right\},
\]
for any $\Psi$ in $\Tcal^\cyl(W')$.

Let then $\Psi \in \Tcal^\cyl(W')$ be arbitrary. Since any $u\in K$ is such that the associated Dirac delta measure $\delta_u$ belongs to $\Pcal(K)$, it follows that
\[ \max_{u\in K} \left\{\phi(u) + \dual{F(u), \Psi'(u)}_{W', W}\right\} = \max_{u\in K} \int_K \left\{\phi(v) + \dual{F(v), \Psi'(v)}_{W', W}\right\} \;\rmd\delta_u(v).
\]

Since
\[ \mu \mapsto \int_{B_\rmw} \left\{\phi(u) + \dual{F(u), \Psi'(u)}_{W', W}\right\} \;\rmd\mu(u)
\]
is a continuous linear function on the compact convex space $\Pcal(K)$, it follows from Bauer's Maximum Principle (see Section \ref{secextremeanddirac}) that this function has a maximizer at an extreme point in $\Pcal(K)$.

Since $K$ is a compact and metrizable topological space and such space is separable, it follows from \cite[Theorem 15.9]{AliBor2006} that the set of extreme points of $\Pcal(K)$ is precisely the set of Dirac delta measures $\delta_u$, for any $u\in K$. Thus, 
\begin{multline*}
  \max_{\mu\in \Pcal(K)} \int_K \left\{\phi(u) + \dual{F(u), \Psi'(u)}_{W', W}\right\} \;\rmd\mu(u) \\
   = \max_{u\in K}  \int_K \left\{\phi(v) + \dual{F(v), \Psi'(v)}_{W', W}\right\} \;\rmd\delta_u(v) \\
   = \max_{u\in K} \left\{\phi(u) + \dual{F(u), \Psi'(u)}_{W', W}\right\},
\end{multline*}
which proves the equality between \eqref{maxmeasinf_to_infmaxmeas} and \eqref{infmaxmeas_to_infmaxpointwise}. 

This completes the proof of Theorem \ref{thmminimaxforsss}.
\end{proof}

\begin{rmk}
  In case the right hand side of \eqref{minimaxforssseq} is minus infinity, this means that $\mathcal{P}_\mathrm{wsss}$ is empty. This is relevant, for instance, if one tries to evaluate the bound on the right using a computer and comes up with a seemingly diverging sequence of numbers going to minus infinity.
\end{rmk}

\begin{rmk}
  There is a more straightforward way to prove that the step \eqref{maxmeasinf_to_infmaxmeas} equals \eqref{infmaxmeas_to_infmaxpointwise}, as noticed by the referee. By taking the maximum of the integrand in \eqref{maxmeasinf_to_infmaxmeas} one obtains the lower than or equal inequality. Then, by considering the Dirac delta on the point of maximum in \eqref{maxmeasinf_to_infmaxmeas}, one obtains the greater than or equal inequality, proving the equality.
\end{rmk}

From Theorem \ref{thmminimaxforsss}, we have the following corollary.
\begin{cor}
  Under the conditions of Theorem \ref{thmminimaxforsss}, if all the weak stationary statistical solutions on $X$ are carried by the compact subset $K$, then
  \begin{equation} 
    \max_{\mu\in\Pcal_\wsss(X)} \int_X \phi(u) \;\rmd\mu(u) = \inf_{\Psi\in \Tcal^\cyl(W')} \max_{u\in K} \left\{\phi(u) + \dual{F(u), \Psi'(u)}_{W', W}\right\}.
  \end{equation}
  for all $\phi\in\Ccal(K)$.  
\end{cor}

If we relax the condition that $F$ is continuous on $K$, we do not obtain an optimal result, but we still have a useful minimax-type upper bound.

\begin{thm}
  \label{thmminimaxuppserboundforsss}
  Suppose $K$ is a compact and metrizable subset of $X$ and assume the set $\Pcal_\wsss(K)$ of weak stationary statistical solutions carried by $K$ is not empty. Let $\phi\in \Ccal(K)$. Then,
  \begin{equation} 
    \label{eqminimaxuppserboundforsss}
    \max_{\mu\in\Pcal_\wsss(K)} \int_K \phi(u) \;\rmd\mu(u) \leq \inf_{\Psi\in \Tcal^\cyl(W')} \sup_{u\in K} \left\{\phi(u) + \dual{F(u), \Psi'(u)}_{W', W}\right\}.
  \end{equation}
\end{thm}

\begin{proof}
  The only step in \eqref{thmminimaxforsss} that uses the continuity of $F$ is the one that uses Sion's Minimax Theorem \ref{thmsionminimax}. Without the continuity, we resort to Proposition \ref{propminimaxupperbound}, which yields only an upper bound, which eventually leads to \eqref{eqminimaxuppserboundforsss}.
\end{proof}

\section{Optimal minimax estimate for continuous semigroups}

In this section, we assume the framework and conditions described in Section \ref{subsecsss}, namely, that $X$ is a Hausdorff space, $W$ is a Banach space, the inclusion $X\subset W'$ is continuous, $F:X\rightarrow W$ is a Borel map, and $\{S(t)\}_{t\geq 0}$ is a semigroup on $X$ associated with equation \eqref{diffeq}. In this context, we first prove the following result.
\begin{prop}
  \label{propminimaxfortimeave}
  Let $B$ be a positively invariant set for $\{S(t)\}_{t\geq 0}$ which is closed in $X$ and is a normal topological space with the topology inherited from $X$.  Suppose there exists a compact and metrizable subset $K$ of $X$ which attracts the points of $B$. Let $\phi \in \Ccal_b(B)$. Then, $\Pcal_\inv(B\cap K)$ is not empty and
  \begin{equation} 
    \label{eqminimaxfortimeave}
    \max_{u_0\in B} \limsup_{T\rightarrow \infty} \frac{1}{T}\int_0^T \phi(S(t)u_0) \;\rmd t = \max_{\mu\in \Pcal_\inv(B\cap K)} \int_B \phi(u) \;\rmd\mu(u).
  \end{equation}
\end{prop}

\begin{proof}
Let $\phi \in \Ccal_b(B)$. For each element $u_0\in B$, it follows from Lemma \ref{kattractingtimeavemeas} applied to $u(t) = S(t)u_0$, $t\geq 0$, that there exists a measure $\mu=\mu_{u_0}$ for which \eqref{genlimandmu}, \eqref{genlimarbitraryvarphi}, and \eqref{genlimparticularphi} hold.

For any given $u_0\in B$ and any $\tau>0$, it follows from \eqref{genlimandmu} and the time-translation invariance property \eqref{timeaveinvgenlim} that, for any $\varphi\in \Ccal_b(B)$, we have $\varphi\circ S(\tau)\in \Ccal_b(B)$, so that
\begin{multline*}
   \int_X \varphi(S(\tau)v) \;\rmd\mu_{u_0}(v) = \Lim \frac{1}{T} \int_0^T \varphi(S(\tau)S(t)u_0)\;\rmd t = \Lim \frac{1}{T} \int_0^T \varphi(S(\tau + t)u_0)\;\rmd t \\ 
   = \Lim \frac{1}{T} \int_0^T \varphi(S(t)u_0)\;\rmd t = \int_X \varphi(v) \;\rmd\mu_{u_0}(v).
\end{multline*}
Since $\mu_{u_0}$ is carried by the compact set $B\cap K$ and since $B$ is normal, any function in $\Ccal(B\cap K)$ can be extended to a continuous and bounded function on $B$, so that 
\[ \int_X \varphi(S(\tau)v) \;\rmd\mu_{u_0}(v) =  \int_X \varphi(v) \;\rmd\mu_{u_0}(v)
\]
holds for any $\varphi\in \Ccal(B\cap K)$. Using again that $B\cap K$ is compact and metrizable and that the weak-star topology in such spaces is Hausdorff (see Section \ref{secprobspaces}), it follows from the previous relation that $S(\tau)\mu_0 = \mu_0$, for any $\tau>0$, which means $\mu_{u_0}$ is an invariant measure for the semigroup. We should remark that an alternative way to obtain this invariant measure is by using the Krylov-Bogolyubov \cite{KrylovBogolyubov1937} approach.

Now, using \eqref{genlimparticularphi} with $\mu=\mu_{u_0}$ and taking the supremum over $u_0\in B$, we find
\begin{multline*}
  \sup_{u_0\in B} \limsup_{T\rightarrow \infty} \frac{1}{T}\int_0^T \phi(S(t)u_0) \;\rmd t  \\ 
  = \sup_{u_0\in B} \int_{B\cap K} \phi(v)\;\rmd\mu_{u_0}(v) \leq \sup_{\mu\in \Pcal_\inv(B\cap K)} \int_{B\cap K} \phi(v)\;\rmd\mu(v).
\end{multline*} 
Since $\Pcal_\inv(B\cap K)$ is a closed set in $\Pcal(B\cap K)$ and $\Pcal(B\cap K)$ is compact, so is $\Pcal_\inv(B\cap K)$. Moreover, the integral on the right hand side above is continuous in $\mu$ on $\Pcal(B\cap K)$. Thus, the supremum is achieved and we have the inequality with a maximum:
\begin{multline}
  \label{ineeqminimaxfortimeave1}
  \sup_{u_0\in B} \limsup_{T\rightarrow \infty} \frac{1}{T}\int_0^T \phi(S(t)u_0) \;\rmd t \leq  \max_{\mu\in \Pcal_\inv(B\cap K)} \int_{B\cap K} \phi(v)\;\rmd\mu(v) \\
  = \int_{B\cap K} \phi(v)\;\rmd\mu^*(v),
\end{multline}
where $\mu^*\in \Pcal_\inv(B\cap K)$ is a maximum point.

Now, it follows from Theorem \ref{thmextremeinvariantmeasareergodic} that $\mu^*$ is ergodic. Hence, 
\begin{equation}
  \label{aeeqminimaxfortimeave}
  \lim_{T\rightarrow \infty} \frac{1}{T} \int_0^T \phi(S(t)v) \;\rmd t = \int_{B\cap K} \phi(v)\;\rmd\mu^*(v)
\end{equation}
holds for $\mu^*$-almost every $v\in B\cap K$. This implies that
\begin{multline*}
  \int_{B\cap K} \phi(v)\;\rmd\mu^*(v) \leq \sup_{v\in B\cap K} \limsup_{T\rightarrow \infty} \frac{1}{T} \int_0^T \phi(S(t)v) \;\rmd t \\
  = \sup_{u\in B} \limsup_{T\rightarrow \infty} \frac{1}{T} \int_0^T \phi(S(t)u) \;\rmd t.
\end{multline*}
Therefore, 
\begin{multline}
  \label{ineeqminimaxfortimeave2}
  \sup_{u\in B} \limsup_{T\rightarrow \infty} \frac{1}{T} \int_0^T \phi(S(t)u) \;\rmd t = \int_{B\cap K} \phi(v)\;\rmd\mu^*(v) \\
  = \max_{\mu\in \Pcal_\inv(B\cap K)} \int_{B\cap K} \phi(v)\;\rmd\mu(v).
\end{multline}
Then, combining \eqref{ineeqminimaxfortimeave1}, \eqref{ineeqminimaxfortimeave2} and \eqref{aeeqminimaxfortimeave}, we prove \eqref{eqminimaxfortimeave}.
\end{proof}

Combining Proposition \ref{propminimaxfortimeave} with Theorem \ref{thmminimaxforsss} yields the main result in this section.
\begin{thm}
Let $B$ be a positively invariant set for $\{S(t)\}_{t\geq 0}$ which is closed in $X$ and is a normal topological space with the topology inherited from $X$.  Suppose there exists a compact and metrizable subset $K$ of $X$ which attracts the points of $B$. Suppose $F$ is continuous on $K$. Let $\phi \in \Ccal_b(B)$. Then,
  \begin{equation} 
    \label{thmminimaxforwsssinsemigroups}
    \max_{\mu\in\Pcal_\wsss(B\cap K)} \int_K \phi(u) \;\rmd\mu(u) = \inf_{\Psi\in \Tcal^\cyl(W')} \max_{u\in B\cap K} \left\{\phi(u) + \dual{F(u), \Psi'(u)}_{W', W}\right\}.
  \end{equation}
  Suppose, further, that $\Pcal_\wsss(B\cap K) = \Pcal_\inv(B\cap K)$. Then
  \begin{multline} 
    \label{thmminimaxfortimeave}
    \max_{u_0\in B} \limsup_{T\rightarrow \infty} \frac{1}{T}\int_0^T \phi(S(t)u_0) \;\rmd t \\ 
      = \inf_{\Psi\in \Tcal^\cyl(W')} \max_{u\in B\cap K} \left\{\phi(u) + \dual{F(u), \Psi'(u)}_{W', W}\right\}.
  \end{multline}
\end{thm}

\begin{proof}
  It follows from Proposition \ref{propminimaxfortimeave} and the condition $\Pcal_\wsss(B\cap K) = \Pcal_\inv(B\cap K)$ that
  \begin{multline}
    \label{thmminimaxforwssswithtimeave}
    \max_{u_0\in B} \limsup_{T\rightarrow \infty} \frac{1}{T}\int_0^T \phi(S(t)u_0) \;\rmd t = \max_{\mu\in \Pcal_\inv(B\cap K)} \int_B \phi(u) \;\rmd\mu(u) \\
  = \max_{\mu\in \Pcal_\wsss(B\cap K)} \int_B \phi(u) \;\rmd\mu(u).
  \end{multline}
  
  This also implies that $\Pcal_\wsss(B\cap K)$ is not empty. Moreover, since $K$ is compact and metrizable and $B$ is closed, the set $B\cap K$ is also compact and metrizable. Thus, Theorem \ref{thmminimaxforsss} applies with $K$ replaced by $B\cap K$, which proves \eqref{thmminimaxforwsssinsemigroups}.
  
  Combining \eqref{thmminimaxforwsssinsemigroups} with \eqref{thmminimaxforwssswithtimeave} proves \eqref{thmminimaxfortimeave}.  
\end{proof}

Again, if we relax the condition that $F$ is continuous on $K$, we do not obtain an optimal result, but we still have a useful minimax-type upper bound. In this case, we can also relax the condition $\Pcal_\wsss(B\cap K) = \Pcal_\inv(B\cap K)$, since all we need is $\Pcal_\inv(B\cap K) \subset \Pcal_\wsss(B\cap K)$, which is valid in general, according to Proposition \ref{propinvimplieswsss}. Then, combining Proposition \ref{propminimaxfortimeave} with the inclusion $\Pcal_\inv(B\cap K) \subset \Pcal_\wsss(B\cap K)$ and with Theorem \ref{thmminimaxuppserboundforsss}, we obtain the following result.

\begin{thm}
  Let $B$ be a positively invariant set for $\{S(t)\}_{t\geq 0}$ which is closed in $X$ and is a normal topological space with the topology inherited from $X$. Suppose there exists a compact and metrizable subset $K$ of $X$ which attracts the points of $B$. Let $\phi \in \Ccal_b(B)$. Then,
  \begin{multline} 
    \label{thmleqminimaxfortimeave}
    \max_{u_0\in B} \limsup_{T\rightarrow \infty} \frac{1}{T}\int_0^T \phi(S(t)u_0) \;\rmd t \\ 
      \leq \inf_{\Psi\in \Tcal^\cyl(W')} \max_{u\in B\cap K} \left\{\phi(u) + \dual{F(u), \Psi'(u)}_{W', W}\right\}.
  \end{multline}
\end{thm}

\section{The Navier-Stokes equations}

Consider the incompressible Navier-Stokes equations 
\begin{equation}
  \label{inseeq}
  \begin{cases}
   \displaystyle \frac{ \partial \bfu}{\partial t} - \nu \Delta \bfu + (\bfu\cdot\bfnabla)\bfu + \bfnabla p = \bff, \medskip \\
    \bfnabla \cdot \bfu = 0,
  \end{cases}
\end{equation}
where $\bfu = \bfu(t,\bfx)\in \RR^d$ is the velocity field; $d=2$ or $3$ is the space dimension; $t\in \RR$ is the time variable; $\bfx=(x,y,z)\in \Omega$ is the space variable; $\Omega\subset \RR^d$ is the domain occupied by the fluid; $p=p(t,\bfx)$ is the kinematic pressure; $\nu>0$ is the kinematic viscosity; and $\bff=\bff(\bfx)$ is the volume density of a time-independent body-force.

For simplicity, we assume the domain $\Omega$ is bounded and has a smooth boundary $\partial \Omega$, and the flow satisfies the no-slip boundary condition $\bfu = 0$ on $\partial \Omega$. Under suitable conditions, different geometries and boundary conditions can be similarly considered, such as fully-periodic flows, periodic channel-flows, or other flows with inhomogeneous boundary conditions after subtracting an appropriate background flow.

We consider the space $H$ of square-integrable, divergence-free vector fields with vanishing normal component on the boundary and endowed with the inner product inherited from the Lebesgue space $L^2(\Omega)^d$. 

We identify $H$ with its dual and consider the intersection of $H$ with the Sobolev space $H_0^1(\Omega)^d$, denoted $V = H\cap H_0^1(\Omega)^d$ , and endowed with the inner product inherited from $H_0^1(\Omega)^d$.

With this setting, the forcing term is assumed to satisfy
\begin{equation} 
  \bff \in V'.
\end{equation}

Further spaces of interest are related to the Stokes operator $\bfA:V\rightarrow V'$ defined by duality according to
\[ \dual{\bfA\bfu,\bfv}_{V',V} = \dual{\bfnabla \bfu, \bfnabla \bfv}_H, \quad \forall \bfu, \bfv\in V.
\]
The Stokes operator is also regarded as a positive symmetric operator $\bfA:D(\bfA)\rightarrow H$ with dense domain $D(\bfA)=\{\bfu\in H; \;\bfA\bfu\in H\}$ and with a compact inverse.

With these spaces, we consider the functional form of the system of equations. We define the bilinear function $B: V \times V \rightarrow V'$ by duality, according to
\[ \dual{\bfB(\bfu, \bfv), \bfw}_{V', V} = \dual{(\bfu\cdot\bfnabla)\bfv, \bfw}_{V',V}, \quad \forall \bfw\in V.
\]
Then, we let $\bfF: V \rightarrow V'$ be given by
\begin{equation} 
  \bfF(\bfu) = \bff - \nu \bfA\bfu - \bfB(\bfu, \bfu).
\end{equation}
With the assumption on $\bff$ and the properties of $\bfA$ and $\bfB(\cdot,\cdot)$, we see that $\bfF$ is continuous from $V$ into $V'$. 

Now, we let $W=D(\bfA) \cap W^{1,\infty}(\Omega)^d$. Using integration by parts and H\"older's inequality, we see that the duality with the bilinear term satisfies, for a suitable constant $C>0$,
\begin{multline}
  \label{bilinestH}
  \left|\int_\Omega ((\bfu\cdot\bfnabla)\bfv) \cdot \bfw \;\rmd \bfx\right| = \left|\int_\Omega ((\bfu\cdot\bfnabla)\bfw) \cdot \bfv \;\rmd \bfx\right| \leq \|\bfu\|_{L^2}\|\bfv\|_{L^2}\|\bfnabla\bfw\|_{L^\infty} \\
  \leq C\|\bfu\|_H\|\bfv\|_H\|\bfw\|_W,
\end{multline}
for any triplet of smooth divergence-free vector fields $\bfu, \bfv, \bfw$. Hence, $\bfB(\cdot,\cdot)$ extends to a \emph{continuous} bilinear form from $H\times H$ into $W'$.

The linear term $\bfu\mapsto \bfA\bfu$ is also continuous from $H$ into $D(A^{-1}) \subset W'$. Finally, since $\bff\in V'\subset W'$, we see that the map $\bfF$ extends to a \emph{continuous} function from $H$ into $W'$. In particular, it is a \emph{Borel map} (see \cite[Corollary 4.26]{AliBor2006}).

\subsection{Weak and Foias-Prodi types of stationary statistical solutions}
\label{secweakandfpsss}

With the choice $W=D(A)\cap W^{1,\infty}(\Omega)^3$ described above, the map $\bfF:H\rightarrow W'$ is continuous in both two- and three-dimensional cases. With such a choice of spaces, and considering $X$ to be the space $H$, it follows from the Definition \ref{defwsss} that a Borel probability measure $\mu$ on $H$ satisfying \eqref{ssseq} with $X=H$ is denoted a \textbf{weak stationary statistical solution} of the Navier-Stokes equations. Following Definition \ref{defwsss}, the space of such solutions in the case of the Navier-Stokes equations is denoted by $\Pcal_\wsss(H)$.

This notion is more general than the definition of a Foias-Prodi stationary statistical solution as defined in \cite{FMRT2001, FRT2019}:
\begin{defs}
  \label{defpsss}
  A \textbf{Foias-Prodi stationary statistical solution} for the Navier-Stokes equations is a Borel probability measure $\mu$ on $H$ with the properties that
\begin{subequations}
  \begin{align}
    \label{finitemeanenstrophy}
    (i) & \int_H \|\bfnabla \bfu\|^2\;\rmd\mu(\bfu) < \infty, \\
    \label{meanenergyineq}
    (ii) & \int_H \psi'(\|\bfu\|^2)\left( \nu \|\bfnabla \bfu\|^2 - \dual{\bff, \bfu}_{V', V} \right) \;\rmd\mu(\bfu) \leq 0, \\
    \label{ssseulerianinvariancecondition}
    (iii) & \int_H \dual{\bfF(\bfu), \Psi'(\bfu)}_{V',V} \;\rmd\mu(\bfu) = 0, \quad \forall \Psi\in \Tcal^\cyl(V'),
  \end{align}
\end{subequations}
for any real-valued continuously-differentiable function $\psi$ which is nonnegative, nondecreasing, and with bounded derivative. We denote the set of Foias-Prodi stationary statistical solutions by $\Pcal_\fpsss(H)$. The subspace of such measures which are carried by a Borel set $E\subset H$ is denoted by $\Pcal_\fpsss(E)$.
\end{defs}

Condition \eqref{finitemeanenstrophy} is that of finite mean enstrophy and guarantees that the term $\bfu \mapsto \dual{\bfF(\bfu), \Psi'(\bfu)}_{V',V}$, in \eqref{ssseulerianinvariancecondition}, is $\mu$-integrable. It is also motivated by the fact that finite-time averages of the enstrophy of Leray-Hopf weak solutions $\bfu(\cdot)$ are bounded, i.e. 
$$ \frac{1}{T} \int_0^T \|\bfnabla \bfu(t)\|^2 \;\mathrm{d}t < \infty.
$$

Condition \eqref{meanenergyineq} is directly motivated by the energy-inequality condition for Leray-Hopf weak solutions. Indeed, if $\bfu(\cdot)$ is such a solution and $\psi$ is as above, then one can show that 
\[ \frac{\mathrm{d}}{\mathrm{d} t}\psi(\|\bfu\|^2) \leq 2\psi'(\|\bfu\|^2)\left( \dual{\bff, \bfu}_{V', V} - \nu \|\bfnabla \bfu\|^2 \right)
\]
in the distribution sense, so that the asymptotic limit $T\rightarrow \infty$ of the time-averages 
\[ \frac{1}{T} \int_0^T \psi'(\|\bfu\|^2)\left( \nu \|\bfnabla \bfu\|^2 - \dual{\bff, \bfu}_{V', V} \right) \;\mathrm{d}t
\]
is nonnegative.

The inclusion $\Pcal_\fpsss(H) \subset \Pcal_\wsss(H)$ follows immediately from the fact that $W \subset V$. The reverse inclusion, on the other hand, might not be true in general. But if the measure $\mu$ is carried by a bounded set in $V$, then the reverse inclusion does hold:

\begin{prop}
  \label{wsssisfpsssin2d}
  If $B$ is a bounded Borel set in $V$, then $\Pcal_\wsss(B) = \Pcal_\fpsss(B)$.
\end{prop}

\begin{proof}
Let $\mu\in \Pcal_\fpsss(B)$. Since $B$ is a subset of $V$ and $\bfF:V\rightarrow V'$, we have that $\bfF(\bfu)\in V'$, for $\mu$-almost every $\bfu$. Then, if $\Psi\in\Tcal^\cyl(W')$ and considering that $W\subset V$ with continuous inclusion, we find that $\dual{\bfF(\bfu), \Psi'(\bfu)}_{W',W} = \dual{\bfF(\bfu), \Psi'(\bfu)}_{V',V},$ for $\mu$-almost every $\bfu$. Thus, \eqref{ssseulerianinvariancecondition} implies \eqref{ssseq}, which means that $\mu\in \Pcal_\wsss(B)$, proving the inclusion $\Pcal_\fpsss(B) = \Pcal_\wsss(B)$

For the converse inequality $\Pcal_\wsss(B) \subset \Pcal_\fpsss(B)$, let $\mu\in \Pcal_\wsss(B)$. For $\mu$ to be in $\Pcal_\fpsss(B)$, we need to show that \eqref{finitemeanenstrophy} and \eqref{meanenergyineq} hold and also that \eqref{ssseulerianinvariancecondition} holds for $\Psi$ in $\Tcal^\cyl(V')$, and not only in $\Tcal^\cyl(W')$.

The first condition \eqref{finitemeanenstrophy} follows easily from the fact that $\mu$ is carried by $B$, which is a bounded set in $V$.

For the condition \eqref{meanenergyineq}, consider the Galerkin projections $P_m$, $m\in\mathbb{N}$, onto the space generated by the first $\bfw_1,\ldots, \bfw_m$ eigenfunctions of the Stokes operator. Take $\Psi_m(\bfu) = \psi(\|P_m\bfu\|^2)$, which is a cylindrical test functional with $\Psi_m'(\bfu) = 2\sum_{j=1}^m \psi'(\|P_m\bfu\|^2) \langle \bfu,\bfw_j\rangle \bfw_j$. Since each $\bfw_j$ belongs to $D(A)\subset W$, we have that $\Psi_m\in \Tcal^\cyl(W')$. Since $\mu$ is a weak stationary statistical solution, it follows from \eqref{ssseq} that
\[ 2\sum_{j=1}^m \int_H \psi'(\|P_m\bfv\|^2) \langle \bfv, \bfw_j\rangle \langle \bfF(\bfv), \bfw_j \rangle_{V',V} \;d\mu(\bfv) = \int_H \dual{\bfF(\bfv),\Psi_m'(\bfv)}_{W', W} \;\rmd \mu(\bfv) = 0,
\]
which means that
\[ \int_H \psi'(\|P_m\bfv\|^2) \langle P_m \bfF(\bfv), P_m \bfv \rangle_{V',V} \;d\mu(\bfv) = 0.
\]
Using that $P_m \bfF(\bfu) = P_m \bff - \nu AP_m \bfu - P_m B(\bfu,\bfu)$, we have, more explicitly, that
\begin{equation}
  \label{galerkinfullstrenghtenineq}
   \int_H \psi'(\|P_m\bfv\|^2) \langle P_m \bff - \nu AP_m \bfu - P_m B(\bfu,\bfu), P_m \bfv \rangle_{V',V} \;d\mu(\bfv) = 0.
\end{equation}

Since $\mu$ is supported by $B$ and $B$ is bounded in $V$, and using that $\langle B(\bfu,\bfu),\bfu\rangle = 0$, we have
\begin{multline*} 
  |\langle P_m B(\bfu,\bfu),P_m \bfu \rangle| = |\langle B(\bfu,\bfu),P_m \bfu \rangle| = |\langle B(\bfu,\bfu),\bfu \rangle - \langle B(\bfu,\bfu),P_m \bfu \rangle| \\
  \leq \|\bfu\|_{L^6}\|\nabla \bfu\|_{L^2}\|\bfu-P_m \bfu\|_{L^3}.
\end{multline*}
In both two and three space dimensions, we have the continuous inclusion $V \subset L^6(\Omega)^d$. Thus, since $B$ is bounded in $V$, there exists a constant $C>0$ such that
\[ |\langle P_m B(\bfu,\bfu),P_m \bfu \rangle| \leq C \|\bfu-P_m \bfu\|_{L^3},
\]
for every $\bfu\in B$ and any $m\in \mathbb{N}$. Using again the assumption that $B$ is bounded in $V$ and using interpolation of $L^3$ between $L^2$ and $H^1$, we have that $\|\bfu-P_m \bfu\|_{L^3}$ goes to zero as $m\rightarrow \infty$, uniformly for $\bfu$ in $B$. This convergence, together with the previous bound, implies that
\begin{equation}
  \label{bilinearpm}
  \int_H \psi'(\|P_m\bfu\|^2) \langle P_m B(\bfu,\bfu), P_m \bfu \rangle_{V',V} \;d\mu(\bfu) \rightarrow 0,
\end{equation}
as $m\rightarrow \infty$.

Now, taking the limit in \eqref{galerkinfullstrenghtenineq} as $m\rightarrow\infty$ and using \eqref{bilinearpm} we find, at the limit, 
\[ \int_H \psi'(\|\bfu\|^2)\left( \nu \|\nabla \bfu\|^2 - \langle \bff, \bfu \rangle_{V', V} \right) \;d\mu(\bfu) = 0,
\]
which proves condition \eqref{meanenergyineq}.

Now, it remains to show that \eqref{ssseq} extends to \eqref{ssseulerianinvariancecondition}. This follows by a density argument. Indeed, let $\Psi\in \Tcal^\cyl(V')$, which must be of the form $\Psi(\bfu) = \psi(\dual{\bfu,\bfw_1}_{V',V}, \ldots, \dual{\bfu,\bfw_m}_{V', V})$, where $\bfw_k\in V$ and $\psi$ is as in Definition \ref{defpsss}. By the density of $W$ in $V$, we approximate, in the topology of $V$, each $\bfw_k$ by a sequence $(\bfw_k^{(j)})_{j\in\mathbb{N}}$ of elements in $W$. This yields a sequence of cylindrical test functionals $\Psi_j(\bfu) = \psi(\dual{\bfu,\bfw_1^{(j)}}_{W',W}, \ldots, \dual{\bfu,\bfw_m^{(j)}}_{W', W})$ in $\Tcal^\cyl(W')$, for which \eqref{ssseq} holds.

Since $\mu$ is carried by $B$, which is a bounded set in $V$, and since $\bfF$ is bounded from $B$ into $V'$, we use the Lebesgue Dominated Convergence Theorem to pass to the limit in $j$ and find that
\[ \int_H \dual{\bfF(\bfu), \Psi'(\bfu)}_{V',V} \;\rmd\mu(\bfu) = \lim_{j\rightarrow\infty} \int_H \dual{\bfF(\bfu), \Psi_j'(\bfu)}_{W',W} \;\rmd\mu(\bfu) = 0.
\]
This completes the proof that $\mu\in\Pcal_\fpsss(B)$.
\end{proof}

The original definition of stationary statistical solution given in \cite{Foias73} and which is used in \cite{FMRT2001}, however, is based on a different, more general set of test functionals (see \cite[Section 6.1, pp. 10--11]{Foias73} for the definition of stationary statistical solution and \cite[Section 3.1.d, p. 249; Section 3.3, pp. 271--273]{Foias72} for that of the space $\mathcal{T}^{\mathrm{ind}}$ of the corresponding test functionals, and also \cite[Definitions IV.1.2, IV.1.3]{FMRT2001}). We denote them, here, by \emph{general test functions}. Since any cylindrical test functional is a general test functional, the cylindrical test functionals give rise to a possibly larger class of stationary statistical solutions. Thus, results with that larger class have a wider scope. It turns out, though, that both definitions are equivalent at least when the measure is carried by a bounded subset of $V$, as discussed below.

\begin{defs}
\label{defgtf}
A \textbf{general test functional} for the Navier-Stokes equations is a functional $\Psi:V\rightarrow \mathbb{R}$ with the following properties:
\begin{enumerate}
  \item $\Psi:V\rightarrow \mathbb{R}$ is continuous from $V$ into $\mathbb{R}$;
  \item \label{gtfdiff} $\Psi$ is Fr\'echet $H$-differentiable in the direction of $V$, i.e. there exists $D\Psi(\bfu)\in H$ such that
     \[ \lim_{\bfv\in V\setminus\{0\}, \;\|\bfv\|_{L^2}\rightarrow 0} \frac{1}{\|\bfv\|_{L^2}} \left|\Psi(\bfu+\bfv) - \Psi(\bfu) - \langle D\Psi(\bfu), \bfv \rangle_{L^2} \right| = 0.
    \]
  \item \label{gtfdiffbdd} $\bfu\mapsto D_\bfu\Psi(\bfu)$ is continuous and bounded from $V$ into $V$.
\end{enumerate}
\end{defs}

Notice, from \eqref{gtfdiff} and \eqref{gtfdiffbdd} in Definition \ref{defgtf}, that, for any general test functional $\Psi$, there exist constants $c_0, c_1>0$ such that
\begin{equation}
  \label{gtfbdd}
  \|\Psi(\bfu)\|_{L^2} \leq c_0 + c_1 \|\bfu\|_{L^2}, \qquad \forall \bfu\in V.
\end{equation}

\begin{prop}
  \label{wsssgeneraltf}
  Let $B$ be a bounded Borel set in $V$ and let $\mu\in\Pcal_\wsss(B)$ be a weak stationary statistical solution of the Navier-Stokes equations in the sense of Definition \ref{defwsss}. Then, \eqref{ssseq} holds for every general test functional as given in Definition \ref{defgtf}.
\end{prop}

\begin{proof}
Let $\Psi$ be a general test functional. The plan is to approximate $\Psi$ by cylindrical test functionals and pass to the limit in \eqref{ssseq}.

Let  $\bfw_j$ be the eigenvectors of the Stokes operator, normalized in $H$, and with corresponding eigenvalues $\lambda_j>0$. Let $P_m$ be the Galerkin projectors over the span of the first $m$ eigenvectors. Let $R>0$ be such that the set $B$ is contained in the ball in $V$ with radius $R$ and centered at the origin, and consider a smooth truncation function $\tau:\mathbb{R}\rightarrow \mathbb{R}$ with compact support, say $\tau(s)=1$ for $|s|\leq 1$, $\tau(s)=0$ for $|s|\geq 4$, and $0\leq \tau(s) \leq 1$, $|\tau'(s)|\leq 1$, for all $s$. Then, define $\Psi_m$, for $m\in \mathbb{N}$, as
\begin{equation}
  \label{defpsim}
  \Psi_m(\bfu) = \tau\left(\frac{\|P_m\bfu\|_V^2}{R^2}\right)\Psi(P_m\bfu).
\end{equation}

In order to see that $\Psi_m$ is a cylindrical test functional, define $\zeta_m:\mathbb{R}^m \rightarrow P_m H$ by
\[ \zeta_m(a) = \sum_{j=1}^m a_j \bfw_j,
\]
where $a=(a_1, \ldots, a_m)\in \mathbb{R}^m$. Then, let
\[ \psi_m(a) = \tau\left(\frac{\|a\|_{m,1}^2}{R^2}\right)\Psi(\zeta_m(a)),
\]
where
\[ \|a\|_{m,1}^2 = \sum_{j=1}^m \lambda_j|a_j|^2.
\]

Notice, now, that $\|a\|_{m, 1} = \|\zeta_m(a)\|_V$ and that $\Psi_m$ can be written as the composition of $\psi$ with the Galerkin projector, i.e $\Psi_m = \psi \circ \zeta_m^{-1}\circ P_m$, or, more explicitly,
\[ \Psi_m(\bfu) = \psi_m(\langle \bfu,\bfw_1\rangle, \ldots, \langle \bfu,\bfw_m\rangle).
\]

By construction, $\psi_m$ is a continuously-differentiable real-valued function, with
\[ \partial_{a_j}\psi_m(a) = \frac{2a_j\lambda_j}{R^2}\tau'\left(\frac{\|a\|_{m,1}^2}{R^2}\right)\Psi(\zeta_m(a)) + \tau\left(\frac{\|a\|_{m,1}^2}{R^2}\right)\langle D_u\Psi(\zeta_m(a)),\bfw_j\rangle_{L^2},
\]
where we used that
\[ \partial_{a_j} \tau\left(\frac{\|a\|_{m,1}^2}{m^2}\right) = \tau'\left(\frac{\|a\|_{m,1}^2}{R^2}\right) \frac{2a_j\lambda_j}{R^2},
\]
and
\[ \partial_{a_j}\Psi(\zeta_m(a)) = \langle D_u\Psi(\zeta_m(a)),\bfw_j\rangle_{L^2}.
\]

Since $\tau(s)$ vanishes for $|s|\geq 4$, we find that $\partial_{a_j}\psi_m(a)$ vanishes for $\|a\|_{m,1} \geq 2R$, for all $j$. Hence, $\psi_m$ has compact support on $\mathbb{R}^m$. This implies that $\Psi_m$ is a cylindrical test functional on $V'$. Thus, \eqref{ssseq} holds for $\Psi_m$.

From the definition of $\Psi_m$, we find that
\begin{equation}
  \label{exprdiffpsim}
  D_\bfu\Psi_m(\bfu) =  \tau'\left(\frac{\|P_m\bfu\|_V^2}{R^2}\right)\Psi(P_m\bfu)\frac{2P_m \bfu}{R^2} + \tau\left(\frac{\|P_m\bfu\|_V^2}{m^2}\right)P_mD_\bfu\Psi(P_m\bfu).
\end{equation}
This means, in particular, that $D_\bfu\Psi_m(\bfu)$ belongs to $V$. Since $\bfF(\bfu)$ belongs to $V'$, for every $\bfu$ in $V$, we also have $\dual{F(v),\Psi_m'(v)}_{V', V} = \dual{F(v),\Psi_m'(v)}_{W', W}$, for every $\bfu\subset B$, which carries $\mu$. Thus,  \eqref{ssseq} holds with $W$ replaced by $V$, i.e.
  \begin{equation}
    \label{ssseqforpsim}
    \int_H \dual{F(v),\Psi_m'(v)}_{V', V} \;\rmd \mu(v) = 0, \qquad \forall m\in\mathbb{N}.
\end{equation}

Now we need to estimate $\|D_\bfu\Psi_m(\bfu)\|_V$. Using \eqref{gtfbdd} and that $\tau$ is bounded by $1$, we find, from \eqref{exprdiffpsim}, that
\begin{multline*}
  \|D_\bfu\Psi_m(\bfu)\|_V \leq  \left|\tau'\left(\frac{\|P_m\bfu\|_V^2}{R^2}\right)\right| \left(c_0 + c_1\|P_m\bfu\|\right)\frac{2\|P_m \bfu\|_V}{R^2} +\|P_mD_\bfu\Psi(P_m\bfu)\|_V \\
  \leq \left|\tau'\left(\frac{\|P_m\bfu\|_V^2}{R^2}\right)\right| \left(c_0 + \frac{c_1}{\lambda_1^{1/2}}\|P_m\bfu\|_V\right)\frac{2\|P_m \bfu\|_V}{R^2} +\|D_\bfu\Psi(P_m\bfu)\|_V.
\end{multline*}
Using that $\tau'(s)$ is also bounded by $1$ and vanishes for $|s|\geq 4$ and that $\|D_\bfu\Psi(\cdot)\|_V$ is bounded, we find
\begin{equation}
  \label{unifboundpsim}
  \|D_\bfu\Psi_m(\bfu)\|_V \leq  \left(c_0 + \frac{2c_1R}{\lambda_1^{1/2}}\right)\frac{4}{R} + c_2,
\end{equation}
for some constant $c_2>0$, which shows that $\|D_\bfu\Psi_m(\bfu)\|_V$ is uniformly bounded in $\bfu\in V$ and in $m\in\mathbb{N}$.

Finally, for any $\bfu\in B$, we have $\|P_m\bfu\|_V^2/R^2 \leq \|\bfu\|_V^2/R^2 \leq 1$. Thus, $\tau(\|P_m\bfu\|_V^2/R^2)$ is identically one, on $B$, and $\tau'(\|P_m\bfu\|_V^2/R^2)$ vanishes. Hence, from \eqref{defpsim} and \eqref{exprdiffpsim}, 
\[ \Psi_m(\bfu) = \Psi(P_m\bfu), \quad D_\bfu\Psi_m(\bfu) =  P_mD_\bfu\Psi(P_m\bfu), \qquad \forall \bfu\in B.
\]
Since $\Psi$ and $D\Psi$ are continuous in $V$ and $P_m\bfu \rightarrow \bfu$ in $V$, as $m\rightarrow\infty$, for each $\bfu\in V$, we find that
\begin{equation}
  \label{pointwiseconvpsim}
  \Psi_m(\bfu) \rightarrow \Psi(\bfu) \textrm{ in } \mathbb{R}, \quad D_\bfu\Psi_m(\bfu) \rightarrow D_\bfu\Psi(\bfu) \textrm{ in } V, \quad \textrm{ as } m\rightarrow \infty, \;\forall \bfu\in B.
\end{equation}

With the pointwise convergence \eqref{pointwiseconvpsim} and the fact that $\mu$ is carried by $B$, we find that the convergences in \eqref{pointwiseconvpsim} hold $\mu$-almost everywhere. With the $\mu$-almost everywhere convergence $D_\bfu\Psi_m(\bfu) \rightarrow D_\bfu\Psi(\bfu) $ and the uniform bound \eqref{unifboundpsim}, we are able to pass to the limit $m\rightarrow \infty$ in \eqref{ssseqforpsim} to prove that \eqref{ssseq} holds for $\Psi$. This completes the proof.
\end{proof}

Therefore, as a consequence of Propositions \ref{wsssisfpsssin2d} and \ref{wsssgeneraltf}, any Foias-Prodi stationary statistical solution in the sense of Definition \ref{defpsss} (and as given in \cite{FMRT2001, FRT2019}) which is carried by a bounded Borel set in $V$ is also a Foias-Prodi stationary statistical solution in the original sense given in \cite{Foias73} (as also used in \cite{FMRT2001}).

\subsection{The two-dimensional case}

In the two-dimensional case, it is well-known that the Navier-Stokes equations generate a continuous semigroup on the space $H$, in the sense of Definition \ref{subsecsemigroups} (see \cite{BabinVishik1992, ConstFoias1989,  Lady1963, SellYou2002, Temam1988}).

With the choice $W=D(\bfA)\cap W^{1,\infty}(\Omega)^2$ described above, the map $\bfF:H\rightarrow W'$ is continuous.

Moreover, it is known that this semigroup possesses a uniformly absorbing set $B_V$ which is a closed ball in $V$, hence compact in $H$ \cite[Section III.2.2]{Temam1988}. Thus, given any closed and invariant set $B$ for the semigroup, it is always possible to find a closed and bounded set $K$ in $V$ which absorbs the points of $B$ (we may always take $K=B_V$, but it may be useful to allow for a smaller set, in case we want to localize the minimax estimate).

With $K$ bounded in $V$, it is proved in \cite[Proposition 6.2]{Foias73}  that Foias-Prodi stationary statistical solutions satisfying  \eqref{ssseq} for general test functional as in Definition \ref{defgtf} are invariant. Thanks to Propositions \ref{wsssisfpsssin2d} and \ref{wsssgeneraltf}, this means that $\Pcal_\wsss(K)\subset \Pcal_\inv(K)$. This, together with Proposition \ref{propinvimplieswsss}, implies that $\Pcal_\inv(K) = \Pcal_\wsss(K)$. Since $B$ is invariant, this means that $\Pcal_\inv(B\cap K)= \Pcal_\wsss(B\cap K)$.

Therefore, Theorem \ref{thmminimaxfortimeave} applies with such $B$ and $K$, and we obtain the following result.
\begin{thm}
  Let $\{S(t)\}_{t\geq 0}$ be the continuous semigroup generated by the two-dimensional Navier-Stokes equations on $H$, let $B$ be a positively invariant set which is closed in $H$, and let $K$ be a closed and bounded subset of $V$ which attracts every point in $B$. Let $\phi \in \Ccal_b(B)$. Then,
  \begin{multline} 
    \label{thmminimaxfortimeave2dnse}
    \max_{u_0\in B} \limsup_{T\rightarrow \infty} \frac{1}{T}\int_0^T \phi(S(t)u_0) \;\rmd t \\ 
      = \inf_{\Psi\in \Tcal^\cyl(W')} \max_{u\in B\cap K} \left\{\phi(u) + \dual{F(u), \Psi'(u)}_{W', W}\right\}.
  \end{multline}
\end{thm}

Since we do not need $K$ to be uniformly attracting, just pointwise attracting, then we can take $B$ to be the whole space $H$, and take $K=B_V$. In this case, we have the following corollary of the above.

\begin{cor}
  Let $\{S(t)\}_{t\geq 0}$ be the continuous semigroup generated by the two-dimensional Navier-Stokes equations. Let $\phi \in \Ccal_b(H)$. Then,
  \begin{multline} 
    \label{corminimaxfortimeave2dnse}
    \max_{u_0\in B} \limsup_{T\rightarrow \infty} \frac{1}{T}\int_0^T \phi(S(t)u_0) \;\rmd t \\ 
      = \inf_{\Psi\in \Tcal^\cyl(W')} \max_{u\in B_V} \left\{\phi(u) + \dual{F(u), \Psi'(u)}_{W', W}\right\}.
  \end{multline}
\end{cor}

\subsection{The three-dimensional case}
\label{sec3dnse}

Concerning weak stationary statistical solutions, Theorem \ref{thmminimaxforsss} yields, in the context of the three-dimensional Navier-Stokes equations, the following result.
\begin{thm}
  \label{thmminimaxfor3Dnsesss}
  Suppose $K$ is a compact subset of $H$ and that the set $\Pcal_\wsss(K)$ of weak stationary statistical solutions carried by $K$ is nonempty. Let $\phi\in \Ccal(K)$. Then,
  \begin{equation} 
    \max_{\mu\in\Pcal_\wsss(K)} \int_K \phi(u) \;\rmd\mu(u) = \inf_{\Psi\in \Tcal^\cyl(W')} \max_{u\in K} \left\{\phi(u) + \dual{\bfF(u), \Psi'(u)}_{W', W}\right\}.
  \end{equation}
\end{thm}

Notice that, since $H$ is a Hilbert space, any compact subset $K$ in $H$ is metrizable, so we do not need to add this condition to $K$, in Theorem \ref{thmminimaxfor3Dnsesss}.

Theorem \ref{thmminimaxfor3Dnsesss} has to be taken with caution, due to the following remark.

\begin{rmk}
  If $\bfu_*$ is a steady solution of the 3D NSE, then the associated Dirac delta measure $\delta_{\bfu_*}$ is a \emph{trivial} stationary statistical solution carried by a compact set, so the statement above is not empty. Periodic and quasi-periodic orbits also lead to stationary statistical solutions carried by compact sets. In general, however, \emph{it is not known} whether \emph{interesting} types of stationary statistical solutions are carried by sets which are compact in $H$. The possible lack of regularity of the solutions prevent us to deduce that all stationary statistical solutions are carried by a compact set in $H$, as usually happens in well-posed dissipative systems.
\end{rmk}
\begin{rmk}
  Another possibility is to consider the space $H_\rmw$, which is the space $H$ endowed with the weak topology, and which seems to be more natural for the study of statistical solutions \cite{FMRT2001, FRT2013, FRT2019}. In particular, it is known that there exists a (weakly) compact set that carries \emph{all} the stationary statistical solutions. However, if we attempt to use $X=H_\rmw$, then we loose the continuity of the bilinear map $\bfB(\cdot, \cdot):X\rightarrow W'$, and the proof above does not work.
\end{rmk}

Concerning upper bound estimates that may not be optimal, we relax the condition that the set be compact and obtain, using Proposition \ref{propminimaxupperbound} instead of Sion's Minimax Theorem \ref{thmsionminimax}, and the general result given by Proposition \ref{propinvimplieswsss}, the following result.
\begin{thm}
  \label{thmminimaxfor3Dfpsss}
  Let $B$ be a bounded Borel set in $H$ such that $\Pcal_\wsss(B)$ is not empty. Let $\phi\in \Ccal(B)$. Then,
  \begin{multline} 
    \sup_{\mu\in\Pcal_\fpsss(B)} \int_B \phi(u) \;\rmd\mu(u) \leq \sup_{\mu\in\Pcal_\wsss(B)} \int_B \phi(u) \;\rmd\mu(u) \\
    \leq \inf_{\Psi\in \Tcal^\cyl(W')} \sup_{u\in B} \left\{\phi(u) + \dual{\bfF(u), \Psi'(u)}_{W', W}\right\}.
  \end{multline}
\end{thm}

A particular type of Foias-Prodi stationary statistical solution is obtained as generalized limits of time averages of Leray-Hopf weak solutions (see \cite{FoiTem75, FMRT2001, FRT2019}). The following result uses this concept. But first we need some definitions.

We say that a set $B\subset H$ is \textbf{positively invariant} with respect to the set of Leray-Hopf weak solutions if, given $\bfu_0\in B$, \emph{any} such solution $\bfu$ with initial condition $\bfu(0)=\bfu_0$ is such that $\bfu(t)\in B$, for all $t\geq 0$. We denote by $\Ucal(B)$ the set of all Leray-Hopf weak solutions with values in $B$. It is known that a sufficiently large ball $B_R$ in $H$ absorbs all the Leray-Hopf weak solutions and, being a bounded ball on a Hilbert space, $B_R$ is compact for the weak topology in $H$ (see \cite{FoiTem75, Temam1988, FMRT2001, FRT2019}). 

Then, the fact that these generalized limits lead to Foias-Prodi stationary statistical solutions yields
\[ \sup_{\bfu\in \Ucal(B)}\limsup_{T\rightarrow \infty} \frac{1}{T} \int_0^T \phi(\bfu(t)) \;\rmd t \leq \max_{\mu\in\Pcal_\fpsss(B\cap B_R)} \int_K \phi(u) \;\rmd\mu(u),
\]
so that, using Theorem \ref{thmminimaxfor3Dfpsss}, we find
\begin{thm}
  \label{thmminimaxfor3Dlhws}
  Let $B\subset H$ be a positively invariant set for the Leray-Hopf weak solutions of the three-dimensional Navier-Stokes equations. Let $B_R$ be a sufficiently large ball in $H$ that absorbs all the Leray-Hopf weak solutions. Let $\phi\in \Ccal(K)$. Then,
  \begin{multline} 
    \sup_{\bfu\in \Ucal(B)}\limsup_{T\rightarrow \infty} \frac{1}{T} \int_0^T \phi(\bfu(t)) \;\rmd t \\
    \leq \inf_{\Psi\in \Tcal^\cyl(W')} \max_{u\in B \cap B_R} \left\{\phi(u) + \dual{\bfF(u), \Psi'(u)}_{W', W}\right\}.
  \end{multline}
\end{thm}

\section*{Acknowledgments}

We thank the referee for his/her thorough assessment of the original version of the manuscript, with a number of comments that helped improve the quality of this paper. This work was partly supported by the CNPq, Bras\'{\i}lia, Brazil, by the National Science Foundation (NSF) under the grant DMS-1510249, and by the Research Fund of Indiana University.

\end{document}